\documentclass[twoside,11pt,reqno]{amsart}
\usepackage{amsmath,amssymb,amscd,mathrsfs,epic,wasysym,latexsym,tikz,mathrsfs,cite,hyperref,tensor,enumerate,graphicx, cleveref}
\usepackage{stmaryrd}
\usepackage{pb-diagram}
\usepackage[matrix,arrow]{xy}

\makeatletter

\numberwithin{equation}{section}

\newtheorem{Proposition}[equation]{Proposition}
\newtheorem{Lemma}[equation]{Lemma}
\newtheorem{Theorem}[equation]{Theorem}
\newtheorem{Corollary}[equation]{Corollary}
\theoremstyle{definition}  
\newtheorem{Definition}[equation]{Definition}

\newtheorem{Conjecture}[equation]{Conjecture}

\let\<\langle
\let\>\rangle

\newcommand\Comment[2][\relax]{\space\par\medskip\noindent%
\fbox{\begin{minipage}{\textwidth}\textbf{Comment\ifx\relax#1\else---#1\fi}\newline%
		#2\end{minipage}}\medskip
}

\def\balpha{\text{\boldmath$\alpha$}}
\def\bbeta{\text{\boldmath$\beta$}}

\def\bl{\text{\boldmath$l$}}

\def\bs{\text{\boldmath$s$}}

\def\bt{\text{\boldmath$t$}}
\def\bc{\text{\boldmath$c$}}
\def\br{\text{\boldmath$r$}}
\def\b1{\text{\boldmath$1$}}

\def\ba{\text{\boldmath$a$}}
\def\bb{\text{\boldmath$b$}}
\def\bw{\text{\boldmath$w$}}
\def\bu{\text{\boldmath$u$}}
\def\bv{\text{\boldmath$v$}}
\def\bx{\text{\boldmath$x$}}

\def\balpha{\text{\boldmath$\alpha$}}
\def\bbeta{\text{\boldmath$\beta$}}
\def\biota{\text{\boldmath$\iota$}}

\def\coproduct{\reflectbox{\rotatebox[origin=c]{180}{${\tt\Delta}$}}}

\def\fa{\mathfrak{a}}
\def\a{\mathfrak{a}}
\def\c{\mathfrak{c}}

\def\pmod#1{\text{ }(\text{\rm mod } #1)\,}

\newcommand{\Hom}{\operatorname{Hom}}

\newcommand{\End}{\operatorname{End}}

\newcommand{\id}{\operatorname{id}}

\newcommand{\Stab}{\operatorname{Stab}}

\newcommand{\Z}{\mathbb{Z}}

\newcommand{\F}{\mathbb{F}}

\newcommand{\0}{{\bar 0}}
\renewcommand{\1}{{\bar 1}}
\def\eps{{\varepsilon}}
\def\phi{{\varphi}}

\newcommand{\zX}{{\mathsf{X}}}
\newcommand{\zY}{{\mathsf{Y}}}
\newcommand{\zB}{{\mathsf{B}}}
\newcommand{\zL}{{\mathsf{L}}}

\newcommand{\zb}{{\mathsf{b}}}

\newcommand{\zc}{{\mathsf{c}}}
\newcommand{\zv}{{\mathsf{v}}}

\newcommand{\ze}{{\mathsf{e}}}

\newcommand{\za}{{\mathsf{a}}}
\newcommand{\zw}{{\mathsf{w}}}

\newcommand{\Ro}{\operatorname{Row}}
\newcommand{\Co}{\operatorname{Col}}

\newcommand{\Ga}{\Gamma}
\newcommand{\la}{\lambda}
\newcommand{\La}{\Lambda}
\newcommand{\al}{\alpha}

\def\Si{\mathfrak{S}}
\newcommand{\si}{\sigma}

\newcommand{\Om}{\Omega}

\newcommand{\de}{\delta}
\newcommand{\De}{\Delta}
\newcommand{\ka}{\kappa}

\def\id{\mathop{\mathrm {id}}\nolimits}

\renewcommand{\mod}{\bmod \,}

\newcommand{\Zig}{{\bar {\mathsf Z}}}

\newcommand{\EZig}{{\mathsf Z}}
\renewcommand{\Alph}{{\mathscr A}}

\def\col{{\operatorname{color}}}

\newcommand{\Std}{\operatorname{Std}}

\def\Brackets#1{[ #1 ]}

\def\b{\mathfrak{b}}
\def\k{\Bbbk}

\def\T{\text{\boldmath$T$}}

\def\Stab{\text{\boldmath$S$}}

\def\spa{\operatorname{span}}

\def\op{{\mathrm{op}}}
\def\sop{{\mathrm{sop}}}

\def\onto{{\twoheadrightarrow}}
\def\into{{\hookrightarrow}}

\def\mod#1{#1\!\operatorname{-mod}}

\def\iso{\stackrel{\sim}{\longrightarrow}}

\def\ch{\operatorname{ch}}
\def\lan{\langle}
\def\ran{\rangle}

\def\Seq{\operatorname{Tri}}
\def\Se{\operatorname{Seq}}

\newcommand{\rev}{\overleftarrow}

\def\bla{\text{\boldmath$\lambda$}}
\def\bmu{\text{\boldmath$\mu$}}
\def\bnu{\text{\boldmath$\nu$}}

{\catcode`\|=\active
\gdef\set#1{\mathinner{\lbrace\,{\mathcode`\|"8000%
			\let|\midvert #1}\,\rbrace}}
}
\def\midvert{\egroup\mid\bgroup}

\colorlet{darkgreen}{green!50!black}
\tikzset{dots/.style={very thick,loosely dotted},
greendot/.style={fill,circle,color=darkgreen,inner sep=1.5pt,outer sep=0},
blackdot/.style={fill,circle,color=black,inner sep=1.5pt,outer sep=0},
graydot/.style={fill,circle,color=gray,inner sep=1.1pt,outer sep=0}
}
\def\greendot(#1,#2){\node[greendot] at(#1,#2){}}
\def\blackdot(#1,#2){\node[blackdot] at(#1,#2){}}
\def\graydot(#1,#2){\node[graydot] at(#1,#2){}}

\newenvironment{braid}{
\begin{tikzpicture}[baseline=6mm,black,line width=1pt, scale=0.32,
	draw/.append style={rounded corners},
	every node/.append style={font=\fontsize{5}{5}\selectfont}]%
}{\end{tikzpicture}
}

\def\Grid(#1,#2){
\draw[very thin,gray,step=2mm] (0,0)grid(#1,#2);
\draw[very thin,darkgreen,step=10mm] (0,0)grid(#1,#2);
}

\newcommand\Tableau[2][\relax]{
\begin{tikzpicture}[scale=0.5,draw/.append style={thick,black}]
	\ifx\relax#1\relax%
	\else 
	\foreach\box in {#1} { \filldraw[blue!30]\box+(-.5,-.5)rectangle++(.5,.5); }
	\fi
	\newcount\row\newcount\col
	\row=0
	\foreach \Row in {#2} {
		\col=1
		\foreach\k in \Row {
			\draw(\the\col,\the\row)+(-.5,-.5)rectangle++(.5,.5);
			\draw(\the\col,\the\row)node{\k};
			\global\advance\col by 1
		}
		\global\advance\row by -1
	}
\end{tikzpicture}
}

\newcommand\YoungDiagram[2][\relax]{
\begin{tikzpicture}[scale=0.5,draw/.append style={thick,black}]
	\ifx\relax#1\relax%
	\else 
	\foreach\box in {#1} {
		\filldraw[blue!30]\box rectangle ++(1,1);
	}
	\fi
	\newcount\row
	\row=0
	\foreach \col in {#2} {
		\draw(1,\the\row)grid ++(\col,1);
		\global\advance\row by -1
	}
\end{tikzpicture}
}

\newcommand{\fc}{\mathfrak{c}}
\newcommand{\Tilt}{\operatorname{T}}
\newcommand{\Tiltz}{\mathsf T}
\newcommand{\scrT}{{\mathscr T}}

\begin{document}

\title[Ringel Duality]{{\bf Ringel Duality for Extended ZigZag Schur Algebra}}

\author{\sc Alexander Kleshchev}
\address{Department of Mathematics\\ University of Oregon\\
	Eugene\\ OR 97403, USA}
\email{klesh@uoregon.edu}

\author{\sc Ilan Weinschelbaum}
\address{Department of Mathematics\\ University of Oregon\\
	Eugene\\ OR 97403, USA}
\email{ilanw@uoregon.edu}



\begin{abstract}
	Extended zigzag Schur algebras are quasi-hereditary algebras which are conjecturally Morita equivalent to RoCK blocks of classical Schur algebras. We prove that  extended zigzag Schur algebras are Ringel self-dual. 
\end{abstract}

\maketitle

\section{Introduction}

Let $S(m,r)$ be a classical Schur algebra over the ground field $\F$ of characteristic $p>0$, see \cite{Green}. A fundamental fact going back to \cite{GreenComb,CPS} is that the algebra $S(m,r)$ is (based) quasi-hereditary. The blocks of $S(m,r)$ are classified in \cite[(2.12)]{Do2}, \cite{Do4}. 

From now on let us assume for simplicity that $r\leq m$. Then the  blocks of $S(m,r)$ and of the symmetric group algebra $\F\Si_r$ are parametrized by pairs $(\rho,d)$ where $\rho$ is a $p$-core and $d$ is a non-negative integer 
such that $|\rho|+pd=r$. Let $B_{\rho,d}$ denote the corresponding block of $S(m,r)$ and $\bar B_{\rho,d}$ denote the corresponding block of $\F\Si_r$. 

A special role in representation theory of $\F\Si_r$ is played by the so-called {\em RoCK blocks} going back to \cite{RoTh,CK}. These are the blocks $\bar B_{\rho,d}$ with $\rho$ satisfying certain combinatorial genericity condition with respect to $d$. The corresponding block $B_{\rho,d}$ of the Schur algebra $S(m,r)$ is then also called {\em RoCK}. The RoCK blocks of symmetric groups are important because they admit a nice `local description.'  Namely, by \cite{EK2}, we have that $\bar B_{\rho,d}$ is Morita equivalent to the zigzag Schur algebra $T^{\Zig}(d,d)$ defined by Turner \cite{T}, see also \cite{EK1}. In view of \cite{CR}, this yields a `local description' of all blocks of symmetric groups up to derived equivalence. 

On the other hand, an arbitrary RoCK block $B_{\rho,d}$ of $S(m,r)$ is {\em conjecturally}\, Morita equivalent to the {\em extended zigzag Schur algebra}\, $T^\EZig(d,d)$, see \cite[Conjecture 7.60]{KMgreen3}. As a first evidence for this conjecture, it is proved in \cite[Theorem 1]{KMgreen3} that $T^\EZig(d,d)$ is quasi-hereditary. In this paper we obtain further evidence for this conjecture in terms of Ringel duality. 

In fact, Donkin \cite[(3.7),(3.11)]{DonkinTilt}, \cite[\S5(2)]{Do4}, \cite[\S4.1]{DonkinQS} proves that $S(m,r)$ is Ringel self-dual. 
It follows from the results of Donkin that the Ringel dual of the block $B_{\rho,d}$ is the block $B_{\rho',d}$. On the other hand, if 
$B_{\rho,d}$ is RoCK then so is $B_{\rho',d}$, so we expect that the extended zigzag Schur algebra $T^\EZig(d,d)$ must be Ringel self-dual. This is what we prove in this paper:

\vspace{2mm}
\noindent
{\bf Main Theorem.}
{\em 
	Let $d\leq n$. Then the extended zigzag Schur algebra $T^\EZig(n,d)$ is Ringel self-dual.		
}
\vspace{2mm}

Here $\EZig$ stands for the extended zigzag algebra corresponding to the quiver with $p$ vertices, see \S\ref{SSZig}. 
In fact, $T^\EZig(n,d)$ is a special case of the generalized Schur algebras $T^A_\a(n,d)$ introduced in \cite{KMgreen2}, with $A=\EZig$ and $\a$ being the degree zero component of the graded algebra $\EZig$. It is well known that $\EZig$ is quasi-hereditary and Ringel self-dual. So our Main Theorem is a special case of the following conjecture.

\vspace{2mm}
\noindent
{\bf Conjecture.}
{\em 
	Let $A$ be a based quasi-hereditary algebra and $d \leq n$. If $A'$ is a Ringel dual of $A$, then a Ringel dual of $T^A_\a(n,d)$ is of the form $T^{A'}_{\a'}(n,d)$ for some canonical choice of~$\a'$.
}
\vspace{2mm}

The paper is organized as follows. Section~\ref{SPrelim} is preliminaries.
In particular, \S\ref{SSBQHA} details necessary facts on based quasi-hereditary algebras and \S\ref{SSComb} is on the combinatorics of partitions and tableaux. 
Section~\ref{SSA} describes the construction of $T^A(n,d)$ and important results about its (co)standard modules.
In Section~\ref{SDiv}, we define the modified divided power which we  will  use extensively to construct a full tilting module for $T^{\EZig}(n,d)$.
Lastly, in Section~\ref{SZig} prove the Main Theorem. In particular, in \S\ref{TZFull} we describe a full tilting module for $T^\EZig(n,d)$; and in \S\ref{TZSelf} we compute its endomorphism algebra.

\section{Preliminaries}

\label{SPrelim}
\subsection{General notation}
\label{SGeneralNot}
For $n\in\Z_{>0}$, we denote $[n]:=\{1,2,\dots,n\}$.
Throughout the paper, $I$ denotes a non-empty finite partially ordered set. We always identify $I$ with the set $\{0,1,\dots,\ell\}$ for $\ell=|I|-1$, so that the standard total order on integers refines the partial order on $I$. For a set $S$, we often write elements of $S^d$ as words $\bs=s_1\cdots s_d$ with $s_1,\dots,s_d\in S$. 

The symmetric group $\Si_d$ acts on the right on $S^d$ by place permutations:
$$
(s_1\cdots s_d)\si=s_{\si 1}\cdots s_{\si d}.
$$
For $\bs=s_1\cdots s_d\in S^d$ we have the stabilizer 
$
\Si_\bs:=\{\si\in\Si_d\mid \bs\si=\bs\}.
$ 
For $\bs,\bt\in S^d$, we write $\bs\sim\bt$ if $\bs\si=\bt$ for some $\si\in \Si_d$. If $S_1,\dots,S_m$ are sets, then $\Si_d$ acts on $S_1^d\times\dots\times S_m^d$ diagonally. We write 
$(\bs_1,\dots,\bs_m)\sim (\bt_1,\dots,\bt_m)$ 
if 
$(\bs_1,\dots,\bs_m)\si=(\bt_1,\dots,\bt_m)$ for some $\si\in \Si_d$. 
If $U\subseteq S_1^d\times\dots\times S_m^d$ is a $\Si_d$-invariant subset, we denote by $U/\Si_d$ a complete set of the $\Si_d$-orbit representatives in $U$ and we identify $U/\Si_d$ with the set of $\Si_d$-orbits on $U$.

An (arbitrary) ground field is denoted by $\F$. 
Often we will also need to work over a characteristic $0$ principal ideal domain $R$ such that $\F$ is a $R$-module, so that we can change scalars from $R$ to $\F$ (in all examples of interest to us, one can use $R=\Z$). When considering $R$-supermodules (or $\F$-superspaces) below, we always consider $R$ (and $\F$) as concentrated in degree $\0$.

We use $\k$ to denote $\F$ or $R$ and use it whenever the nature of the ground ring is not important. On the other hand, when it is important to emphasize whether we are working over $R$ or $\F$, we will use lower indices; for example for an $R$-algebra $A_R$ and an $A_R$-module $V_R$, after extending scalars we have $V_\F:=\F\otimes_RV_R$ is a module over $A_\F:=\F\otimes_RA_R$. 

\subsection{Superalgebras and supermodules}
\label{SSSpaces}
Let $V=\bigoplus_{\eps\in\Z/2}V_\eps$ be a $\k$-supermodule. 
If $v\in V_\eps\setminus\{0\}$ for $\eps\in\Z/2$, we say $v$ is {\em homogeneous}, we write $|v|=\eps$, and we refer to $\eps$ as the {\em parity} of $v$. 
If $S\subseteq V$, we denote $S_\0:=S\cap V_\0$ and $S_\1:=S\cap V_\1$. 
If $W$ is another $\k$-supermodule, the set of all $\k$-linear homomorphisms $\Hom_\k(V,W)$ is a $\k$-supermodule such that for $\eps\in\Z/2$, we have 
$$
\Hom_\k(V,W)_\eps=\{f\in\Hom_\k(V,W)\mid |f(v)|=|v|+\eps\ \text{for all homogeneous $v$}\}.
$$ 

The group $\Si_d$ acts on $V^{\otimes d}$ on the right by   automorphisms, such that for all homogeneous $v_1,\dots,v_d\in V$ and $\si\in \Si_d$, we have 
\begin{equation}\label{ESiAct}
	(v_1\otimes\dots\otimes v_{d})^\si=
	(-1)^{\lan\si;\bv\ran} v_{\si1}\otimes\dots\otimes v_{\si d},
\end{equation}
where, setting $\bv:=v_1\cdots v_d\in V^d$, we have put:
\begin{equation}\label{EAngleSi}
	\lan\si;\bv\ran:=\sharp\{(k,l)\in[d]^2\mid k<l,\  \si^{-1}k>\si^{-1}l,\ |v_k|=|v_l|=\1\}.
\end{equation}
We consider the {\em $d$th divided power} $\Ga^dV$, which by definition is the subspace of invariants 
\begin{equation}\label{EGa}
	\Ga^dV:=\{w\in V^{\otimes d}\mid w^\si=w\ \text{for all $\si\in\Si_d$}\}.
\end{equation}
Let $0\leq c\leq d$. 
Given $w_1\in V^{\otimes c}$ and $w_2\in V^{\otimes (d-c)}$, we define 
\begin{equation}\label{EStarNotationGen}
	w_1* w_2:=\sum_{\si}(w_1\otimes w_2)^\si\in V^{\otimes d},
\end{equation}
where the sum is over all shortest coset representatives $\si$ for $(\Si_c\times \Si_{d-c})\backslash\Si_{d}$.

Let $V$ and $W$ be $\k$-supermodules, $d\in\Z_{\geq 0}$, and let $\bv=v_1\cdots v_d\in V^d$, $\bw=w_1\cdots w_d\in W^d$ be $d$-tuples of homogeneous elements. We denote 
\begin{equation} 
	\label{EAngle2}\lan \bv, \bw\ran:=\sharp\{(k,l)\in[d]^2\mid k>l,\  |v_k|=|w_l|=\1\}.
\end{equation}

Let $A$ be a (unital) $\k$-superalgebra and $V,W$ be $A$-supermodules. 
A {\em homogeneous  $A$-supermodule homomorphism} $f:V\to W$ is a homogeneous $\k$-linear map $f:V\to W$ satisfying $f(av)=(-1)^{|f||a|}af(v)$ for all (homogeneous) $a,v$. For $\eps\in\Z/2$, let $\Hom_A(V,W)_\eps$ be the set of all homogeneous $A$-supermodule homomorphisms of parity $\eps$, and let 
$$\Hom_A(V,W):=\Hom_A(V,W)_\0\oplus \Hom_A(V,W)_\1.$$ We refer to the elements of $\Hom_A(V,W)$ as the 
$A$-supermodule homomorphisms from $V$ to $W$. 
We denote by $\mod{A}$ the category of all finitely generated (left) $A$-supermodules and all $A$-supermodule homomorphisms. We denote by `$\cong$' an isomorphism in this category and by `$\simeq$' an {\em even isomorphism} in this category.

We have the parity change functor $\Pi$ on $\mod{A}$: 
for $V\in\mod{A}$ we have $\Pi V\in\mod{A}$ with $(\Pi V)_\eps=V_{\eps+\1}$ for all $\eps\in \Z/2$ and the new action 
$a\cdot v=(-1)^{|a|}av$ for $a\in A, v\in V$. We have $V\cong \Pi V$ via the identity map. 

Suppose there is an even superalgebra anti-involution $\tau : A \to A$. In particular, 
$\tau (ab) = (-1)^{|a||b|}\tau (b) \tau (a)$  
for all $a,b\in A$. Then $\tau$ is an isomorphism $A\to A^{\sop}$, where the multiplication in $A^{\sop}$ is defined as $a\cdot b:=(-1)^{|a||b|}ba$.  
If $V\in \mod{A}$ then the $\tau$-dual $V^\tau\in\mod{A}$ is the dual $V^*$ as a $\k$-supermodule considered as a left $A$-supermodule via $(a f)(v):=(-1)^{|a||f|}f(\tau(a)v)$ for $a\in A,f\in V^*, v\in V$. 

As usual, the tensor product $A^{\otimes d}$ is a superalgebra with respect to 
$$
(a_1\otimes\dots\otimes a_d)(b_1\otimes\dots b_d)=(-1)^{\lan\ba,\bb\ran}a_1b_1\otimes\dots\otimes a_db_d,
$$
where we have put $\ba:=a_1\cdots a_d$, $\bb:= b_1\cdots b_d$ (here and below, in expressions like this, we assume that all elements are homogeneous). If $V$ is an $A$-supermodule then $V^{\otimes d}$ is a supermodule over $A^{\otimes d}$ with respect to 
$$
(a_1\otimes\dots\otimes a_d)(v_1\otimes\dots v_d)=(-1)^{\lan\ba,\bv\ran}a_1v_1\otimes\dots\otimes a_dv_d,
$$
where we have again put $\ba:=a_1\cdots a_d$, $\bv:= v_1\cdots v_d$. 

The divided power $\Ga^dA$ is a subsuperalgebra of $A^{\otimes d}$. If $V$ is an $A$-supermodule then 
$$\big((a_1\otimes\dots\otimes a_d)(v_1\otimes\dots v_d)\big)^\si=(a_1\otimes\dots\otimes a_d)^\si (v_1\otimes\dots\otimes v_d)^\si$$ 
for all $a_1,\dots,a_d\in A,\,v_1,\dots,v_d\in V,\,\si\in\Si_d$. So $\Ga^dV$ is a subsupermodule of the restriction of $V^{\otimes d}$ to $\Ga^dA$. Thus we will always consider $\Ga^dV$ as a   $\Ga^dA$-supermodule.



\subsection{Based quasi-hereditary algebras}
\label{SSBQHA}
The main reference here is \cite{KMgreen1}. 
Let $A$ be a $\k$-superalgebra.

\begin{Definition} \label{DCC} {\rm \cite{KMgreen1}} 
	{\rm 
		Let $I$ be a finite partially ordered set and let $X=\bigsqcup_{i\in I}X(i)$ and $Y=\bigsqcup_{i\in I}Y(i)$ be finite sets of homogeneous elements of $A$ with distinguished elements $e_i\in X(i)\cap  Y(i)$ for each $i\in I$. 
		For each $i\in I$, we set 
$$
A^{>i}:=\spa\{xy\mid j>i,\, x\in X(j),\, y\in Y(j)\}.
$$ 
We say that $I,X,Y$ is  {\em heredity data} if the following axioms hold: 
		\begin{enumerate}
			\item[{\rm (a)}] $B:=\{x y \mid i\in I,\, x\in X(i),\, y\in Y(i)\}$ is a basis of $A$; 
			
			\item[{\rm (b)}] For all $i\in I$, $x\in X(i)$, $y\in Y(i)$ and $a\in A$, we have
			$$
			a x \equiv \sum_{x'\in X(i)}l^x_{x'}(a)x' \pmod{A^{>i}}
			\ \ \text{and}\ \ 
			ya \equiv \sum_{y'\in Y(i)}r^y_{y'}(a)y' \pmod{A^{>i}}
			$$
			for some $l^x_{x'}(a),r^y_{y'}(a)\in\k$;

			\item[{\rm (c)}] For all $i,j\in I$ and $x\in X(i),\ y\in Y(i)$ we have  
			\begin{align*}
				&xe_i= x,\ e_ix= \de_{x,e_i}x,\ e_i y= y,\ ye_i= \de_{y,e_i}y, 
				\\
				&e_jx=x\ \text{or}\ 0,\ ye_j=y\ \text{or}\ 0. 
			\end{align*}
		\end{enumerate}
			}
\end{Definition}

\vspace{2mm}
If $A$ is endowed with heredity data $I,X,Y$, we call $A$ {\em based quasi-hereditary}, and refer to $B$ as a {\em heredity basis} of $A$. By (c), $e_i^2 = e_i$ for all $i \in I$, so from now on we call $\{e_i \mid i \in I \}$ the {\em standard idempotents} of the heredity data. We set
\begin{align}
B_\a&:=\{xy\mid i\in I, x\in X(i)_\0,y\in Y(i)_\0\}, \
\label{EBA}
\\
B_\c&:=\{xy\mid i\in I, x\in X(i)_\1,y\in Y(i)_\1\},
\label{EBC}
\end{align}
so that 
\begin{equation}\label{EBABC}
	B_\0=B_\a\sqcup B_\c.
\end{equation} 
The heredity data $I,X,Y$ of $A$ is called {\em conforming} if $B_\a$ spans a unital subalgebra of $A$.

Let $A$ be a  based quasi-hereditary superalgebra with heredity data $I,X,Y$ (not necessarily conforming). By \cite[Lemma 3.3]{KMgreen1}, $A$ is quasi-hereditary in the sense of Cline, Parshall and Scott and $\mod{A}$ is a highest weight category (see \cite[Theorem 3.6]{CPS}).
The corresponding standard and costandard modules are defined as follows. Let $i \in I$. Note that $A^{>i}$ the ideal of $A$ generated by $\{ e_j \mid j >i \}$, and denote 
$\tilde A:=A/A^{>i}$, $\tilde a:=a+A^{>i}\in \tilde A$ for $a\in A$.  
The {\em standard $A$-module of highest weight $i$} is defined by $\De(i):=\tilde A \tilde e_i$, which is a free $\k$-module with basis
$\{v_x := \tilde x\mid x\in X(i) \}$, see \cite[\S 2.3]{KMgreen1}.
We also have the {\em right standard} $A$-module $\De^\op(i):= \tilde e_i \tilde A$, and by symmetry every result we have about $\De(i)$ has its right analogue for $\De^\op(i)$, for example $\De^\op(i)$ is a free $\k$-module with basis $\{w_y:=\tilde y\mid y\in Y(i)\}$. 
The {\em costandard $A$-module of highest weight $i$} is defined  by $\nabla(i) := \De^\op(i)^*$ 
with $(a f) (v) = (-1)^{|a||f|+|a||v|} (v a)$ for $a \in A, f\in \De(i)^*, v \in \De(i)$.

Let $V\in\mod{A}$. A {\em standard filtration} 
of $V$ is an $A$-supermodule filtration 
$0=W_0\subseteq W_1\subseteq \dots\subseteq W_l=V$ 
such that for every $r=1,\dots,l$, we have $W_r/W_{r-1}\cong \De(i_r)$ 
for some $i_r\in I$. 
We refer to $\De(i_1),\dots, \De(i_l)$ as the factors of the filtration, and to $\De(i_1)$ (resp. $\De(i_l)$) as the bottom (resp. top) factor.
If $\k = \F$, $i\in I$ and $V$ has a standard filtration as above,  then $\sharp\{1\leq r\leq l\mid W_r/W_{r-1}\cong \De(i)\}$ does not depend on the choice of the standard filtration and is denoted $(V:\De(i))$. In fact, by \cite[Proposition A2.2]{DonkinQS}, we have
\begin{equation}\label{EDeMult}
	(V:\De(i))=\dim\Hom_A(V,\nabla(i)).
\end{equation}
A {\em costandard filtration} \
of $V$ is an $A$-supermodule filtration 
$0=W_0\subseteq W_1\subseteq \dots\subseteq W_l=V$ such that for every  $r=1,\dots,l$, we have $W_r/W_{r-1} \cong \nabla(i_r)$ for some $i_r\in I$.

Let $\Tilt\in\mod{A}$. We say that $\Tilt$ is a {\em tilting supermodule} if it has standard and costandard filtrations. We refer to \cite[\S4]{Rouquier} for the integral version of the tilting theory. 
In particular, by \cite[Propositions 4.26, 4.27]{Rouquier}, for every $i\in I$ there exists a (unique up to isomorphism) indecomposable tilting supermodule $\Tilt(i)$ such that $\De(i)\subseteq \Tilt(i)$ and $\Tilt(i)/\De(i)$ has a standard filtration with factors of the form 
$\De(j)$ for $j< i$; moreover, for every tilting supermodule $\Tilt$ we have   
$
\Tilt
\cong \bigoplus_{i\in I}\Tilt(i)^{\oplus m_i}.
$ 
In this case $\Tilt$ is called a {\em full tilting} supermodule if $m_i>0$ for all $i\in I$. If $\Tilt$ is full tilting, the superalgebra
$
A':=\End_{A}(\Tilt)^\sop
$
is called a {\em Ringel dual} of $A$. 
The algebra $A'$ is defined uniquely up to Morita superequivalence and is quasi-hereditary, see \cite[Proposition 4.26]{Rouquier}. 

\subsection{Multipartitions and tableaux}\label{SSComb}
For a partition $\la$, we have the conjugate partition $\la'$, see \cite[p.2]{Mac}. 
For partitions $\la,\mu,\nu$, we denote by $c^{\,\la}_{\mu,\nu}$ the corresponding Littlewood-Richardson coefficient, see \cite[\S\,I.9]{Mac}. The Young diagram of $\la$ is $[\la]:=\{(r,s)\in \Z_{>0}\times\Z_{>0} \mid s\leq \la_r\}$. We refer to $(r,s)\in[\la]$ as the {\em nodes} of $\la$. 

Let $n\in \Z_{>0}$. We denote $\La(n):=\Z_{\geq 0}^n$ and interpret it as the set of {\em compositions} $\la=(\la_1,\dots,\la_n)$ with $n$ non-negative parts. 
For such $\la$, we set $|\la|:=\la_1+\dots+\la_n$. The partitions with at most $n$ parts are identified with 
$$\La_+(n):=\{\la=(\la_1,\dots,\la_n)\in\La(n)\mid\la_1\geq\dots\geq\la_n\}.$$ 
For $d\in\Z_{\geq 0}$, we let 
$$
\La(n,d):=\{\la\in\La(n)\mid|\la|=d\}\qquad\text{and}\qquad
\La_+(n,d):=\{\la\in\La_+(n)\mid|\la|=d\}.
$$ 
For $\la,\mu\in\La(n)$, we define $\la+\mu:=(\la_1+\mu_1,\dots,\la_n+\mu_n)$.

Recall that $I$ denotes a finite poset. 
We will consider the set of $I$-{\em multicompositions}
\begin{equation}\label{ELaX}
	\La^I(n):=\La(n)^I=\{\bla=(\la^{(i)})_{i\in I}\mid \la^{(i)}\in\La(n)\ \text{for all $i\in I$}\}.
\end{equation} 
For \(\bla,\bmu\in\La^I(n)\) we define \(\bla+\bmu\) to be $\bnu\in\La^I(n)$ with \(\nu^{(i)} = \la^{(i)} + \mu^{(i)}\) for all \(i \in I\). For $d\in\Z_{\geq 0}$, we have the sets of  $I$-multicompositions and $I$-multipartitions of $d$: 
\begin{align*}
\La^I(n,d)&:=\{\bla\in\La^I(n)\mid \sum_{i\in I}|\la^{(i)}|=d\},
\\
\La_+^I(n,d)&:=\{\bla\in\La^I(n,d)\mid \la^{(i)}\in\La_+(n)\ \text{for all $i\in I$}\}.
\end{align*}
For $\bla\in\La^I_+(n,d)$, 
we define $[\bla]:=\bigsqcup_{i\in I}[\la^{(i)}]$ and $\|\bla\|:=(|\la^{(i)}|)_{i\in I}\in\Z_{\geq 0}^I$. 

Via our identification $I=\{0,\dots,\ell\}$, for $\bla=(\la^{(i)})_{i\in I}\in\La^I(n)$, we also write  $\bla=(\la^{(0)},\dots,\la^{(\ell)})$. 
For $i\in I$, and $\la\in\La(n,d)$, define 
\begin{equation}\label{EIota}
	\biota_i(\la):=(0,\dots,0,\la,0,\dots,0)\in\La^I(n,d),
\end{equation}
with $\la$ in the $i$th position. 
We will slightly abuse this notation writing $\biota_i(d)$ for $\biota_i\big((d,0,\dots,0)\big)$ and $\biota_i(1^d)$ for $\biota_i\big((1,\dots,1,0,\dots,0)\big)$.

Let $\leq$ be the partial order on $I$. We have a partial order $\unlhd_I$ on the set $\Z_{\geq 0}^I$ with $(a_i)_{i\in I}\unlhd_I(b_i)_{i\in I}$ if and only if $\sum_{j\geq i}a_j\leq \sum_{j\geq i}b_j$ for all $i\in I$. Let $\unlhd$ be the usual {\em dominance partial order} on $\La(n,d)$, i.e. 
$
\la\unlhd \mu$ if and only if $\sum_{r=1}^s\la_r\leq \sum_{r=1}^s\mu_r$ for all $s=1,\dots,n$.
We define a partial order $\leq_I$ on $\La^I(n,d)$ via:
$\bla\leq_I \bmu$ whenever $\|\bla\|\lhd_I\|\bmu\|$, 
or $\|\bla\|=\|\bmu\|$ and $\la^{(i)}\unlhd \mu^{(i)}$ for all $i\in I$. 

Let $I,X,Y$ be a heredity data on a $\k$-superalgebra $A$ as in \S\ref{SSBQHA}. We introduce {\em colored alphabets}   
$
\Alph_{X}:=[n]\times X \quad\text{and}\quad  \Alph_{X(i)}:=[n]\times X(i).
$
An element $(l,x)\in \Alph_{X}$ is often written as $l^x$. 
If  $L=l^x\in \Alph_{X}$, we denote $\col(L):=x.$
For all $i\in I$, we fix arbitrary total orders `$<$' on the sets $\Alph_{X(i)}$ such that whenever $r < s$ (in the standard order on $[n]$), $r^x < s^x$ for all $x \in X(i)$.

Let $\bla=(\la^{(0)},\dots,\la^{(l)})\in\La^I_+(n,d)$. 
Let $N_1,\dots,N_d$ be the nodes of $[\bla]=[\la^{(0)}]\sqcup\dots\sqcup [\la^{(\ell)}]$ listed along the rows of $[\la^{(0)}]$ from left to right starting from the first row and going down, then along the rows of $[\la^{(1)}]$ from left to right starting with the first row and going down, etc.  A function $\T:\Brackets\bla\to \Alph_{X}$ is called a {\em standard $X$-colored $\bla$-tableau} if the following conditions hold:
\begin{enumerate}
\item[$\bullet$] $\T(\Brackets{\la^{(i)}})\subseteq \Alph_{X(i)}$ for all $i\in I$;
\item[$\bullet$] if $r<s$ and $N_r,N_s$ are in the same row of $\Brackets{\la^{(i)}}$, then $\T(N_r) \leq \T(N_s)$ with $\T(N_r) = \T(N_s)$ allowed only if $\col(\T(N_r))\in X(i)_\0$;

\item[$\bullet$] if $r<s$ and $N_r,N_s$ are in the same column of $\Brackets{\la^{(i)}}$, then $\T(N_r) \leq \T(N_s)$ with $\T(N_r) = \T(N_s)$ allowed only if $\col(\T(N_r))\in X(i)_\1$;
\end{enumerate}
We denote by $\Std^X(\bla)$ the set of all standard $X$-colored $\bla$-tableaux. For $\T\in\Std^{X}(\bla)$, letting $\T(N_r)=l_r^{x_r}\in \Alph_{X}$ for $r=1,\dots,d$, we denote $\bl^\T:=l_1\cdots l_{d}$ and $\bx^\T:=x_1\cdots x_{d}$. 
For $\bmu\in\La^I(n,d)$, we say that $\T$ has {\em left weight} $\bmu$ if there exist $i_1,\dots,i_d\in I$ such that $e_{i_1}x_1=x_1,\,\dots,\,
e_{i_d}x_d=x_d$ and 
$
\bmu=\sum_{c=1}^d \biota_{i_c}(\eps_{l_c}).
$ We have the set of all standard $\bla$-tableaux of left weight $\bmu$: 
\begin{equation}\label{EAl}
	\Std^X(\bla,\bmu):=\{\T\in \Std^X(\bla)\mid \T\ \text{has left weight}\ \bmu\}.
\end{equation}

\section{Modified divided powers}
\label{SDiv}

Throughout the section, we fix a non-negative integer $d$.

\subsection{\boldmath Calibrated supermodules and their modified divided powers}
\label{SSCalMod}
Let 
$V_R=V_{R,\0}\oplus V_{R,\1}$ be a free $R$-supermodule of finite rank. We say that $V_R$ is {\em calibrated} if we are given a decomposition $V_{R,\0}=V_{R,\a}\oplus V_{R,\c}$ into two free $R$-modules. 

Let $V_R$ be a calibrated $R$-supermodule. We choose bases  $B^V_\a$, $B^V_\c$, $B^V_\1$ of $V_{R,\a}, V_{R,\c}, V_{R,\1}$, respectively. Thus $B^V_\0:=B^V_\a\sqcup B^V_\c$ is a basis of $V_{R,\0}$ and  $B^V=B^V_\a\sqcup B^V_\c\sqcup B^V_{\1}$ is a basis of $V_R$. Fix an arbitrary total order $<$ on $B^V$. 
Let $\bb\in (B^V)^d$. We define
\begin{equation}\label{EBBSign}
\lan\bb\ran
:=\sharp\{(k,l)\in[d]^2\mid k<l,\ b_k,b_l\in B^V_\1,\ b_k>b_l\}.
\end{equation}
For any $b\in B^V$, set 
\begin{equation}\label{EMultiplicity}
[\bb:b]:=\sharp\{k\in[d]\mid  b_k=b\} 
\end{equation} 
and let  
\begin{equation}\label{EBBFactorial}
[\bb]^!_{\c} :=\prod_{ b\in B^V_\c}[\bb:b]!.
\end{equation}

Define $\Se(B^V,d)$ to be the set of all $d$-tuples $\bb=b_1\cdots b_d\in (B^V)^d$ such that $b_k=b_l$ for 
some $1\leq k\neq l\leq d$ 
only if $b_k\in B^V_\0$. Then $\Se(B^V,d)\subseteq (B^V)^d$ is a $\Si_d$-invariant subset, so we can choose a corresponding set $\Se(B^V,d)/\Si_d$ of $\Si_d$-orbit representatives and identify it with the set of all $\Si_d$-orbits on $\Se(B^V,d)$, cf.  \S\ref{SGeneralNot}. 

Recall the divided power $R$-supermodule $\Gamma^d V_R$ from  \S\ref{SSSpaces}. 
For $\bb=b_1\cdots b_d\in\Se(B^V,d)$, we have elements 
\begin{equation*}\label{EXiDefV}
x_\bb:= \sum_{\bb'=b_1'\cdots b_d'\sim\bb} 
(-1)^{\lan\bb\ran+\lan\bb'\ran}
b_1'\otimes\dots\otimes b_d'
\in \Ga^d V_R
\qquad
\text{and}
\qquad
y_\bb:=[\bb]^!_{\c}\, x_\bb \in \Ga^d V_R.
\end{equation*}
Define the {\em modified divided power}
$$
\tilde\Ga^d V_R:=\spa_R\{y_\bb\mid \bb\in\Se(B^V,d)\}\subseteq \Ga^d V_R.
$$ 
Note that $\{x_\bb\mid \bb\in\Se(B^V,d)/\Si_d\}$
is a basis of $\Ga^d V_R$ and $\{y_\bb\mid \bb\in\Se(B^V,d)/\Si_d\}$
is a basis of $\tilde\Ga^d V_R$, cf. \cite[(3.9)]{EK1}.

The $n=1$ case of \cite[Proposition 4.11]{KMgreen2} yields:

\begin{Lemma}\label{IndC}
$\tilde\Ga^d V_R$ depends only on $V_{R,\a}$, and not on $V_{R,\c}$ or choice of basis $B^V$.
\end{Lemma}

If $V_R,W_R$ are calibrated $R$-supermodules, then $V_R\oplus W_R$ is also a calibrated $R$-supermodule with $(V_R\oplus W_R)_\a:=V_{R,\a}\oplus W_{R,\a}$ and $(V_R\oplus W_R)_\c:=V_{R,\c}\oplus W_{R,\c}$. Moreover, recalling the star product from (\ref{EStarNotationGen}). 

\begin{Lemma}\label{GaTiSumOld}
We have an isomorphism of $R$-supermodules 
\[ \bigoplus_{d_1 + d_2 = d} (\tilde \Ga^{d_1} V_R)  \otimes (\tilde \Ga^{d_2} W_R)\iso \tilde \Ga^d (V_R \oplus W_R),\ y\otimes y'\mapsto y*y' . \]
\end{Lemma}
\begin{proof}
Note that $y_{\bb^1}*y_{\bb^2}=y_{\bb^1\bb^2}$ for all $\bb^1\in \Se(B^V,d_1),\,\bb^2\in \Se(B^W,d_2)$ and compare bases.  
\end{proof}

\subsection{\boldmath Bilinear form on $\tilde \Ga^dV$}

Let $V_R$ be a calibrated $R$-supermodule.  
Suppose in addition that we are given an ($R$-valued) even supersymmetric (or superantisymmetric), non-degenerate bilinear form $(\cdot, \cdot)$ on $V_R$,
such that $(V_{R,\a},V_{R,\a})=0$ and the $R$-complement $V_{R,\c}$ of $V_{R,\a}$ in $V_{R,\0}$ can be chosen so that the restriction of $(\cdot,\cdot)$ to $V_{R,\a}\times V_{R,\c}$ is a perfect pairing. Until the end of this subsection, we always assume that the complement $V_{R,\c}$ has this property. 

Since the form is non-degenerate we may select bases
$B^V_{\a} = \{a_1, \ldots, a_r\}, B^V_{\c} = \{c_1, \ldots, c_r\}$ and $B^V_\1$ for $V_{R, \a}, V_{R, \c}$, and $V_{R, \1}$ respectively
such that $(a_i, c_j) = \de_{i,j}$ for all $i,j \in [r]$. Set $B^V = B^V_{\a} \cup B^V_{\c} \cup B^V_\1$. Let $(B^{V})^*=\{b^*\mid b\in B^V\}$ be the dual basis with respect to $(\cdot, \cdot)$. 
Note that $c_i^* = a_i$ for all $i$, but it is not necessarily true that $a_i^* = c_i$. We take $V_{R,\c}'$ to be the $R$-span of $a_1^*,\dots, a_r^*$, so that $V_{R,\c}'$ is another $R$-complement of $V_{R,\a}$ in $V_{R,\0}$. We now have $(B^{V})^*=(B^{V})^*_\a\cup (B^{V})^*_\c\cup (B^{V})^*_\1$ where 
$$
(B^{V})^*_\a:=B^{V}_\a,\ (B^{V})^*_\c:=\{a_1^*,\dots,a_r^*\},\ (B^{V})^*_\1=\{b^*\mid b\in B^V_\1\}. 
$$
By Lemma~\ref{IndC}, $\tilde\Ga^dV_R$ is independent of the choice of a complement $V_{R,\c}$ and of the choice of a  corresponding basis. So we now have two $R$-bases of 
$\tilde\Ga^dV_R$: 
$$
\{y_\bb  \mid \bb\in\Se(B^V,d)/\Si_d\}
\quad
\text{and}
\quad 
\{y_\bb \mid \bb\in\Se((B^{V})^*,d)/\Si_d\}.
$$

The form $(\cdot,\cdot)$ extends to the form $(\cdot,\cdot)_d$ on $V_R^{\otimes d}$ such that 
\begin{equation}\label{TBil}
(v_1 \otimes \dots \otimes v_d, w_1 \otimes \dots \otimes w_d)_d = (-1)^{\lan \bv, \bw \ran} (v_1, w_1) \cdots (v_d, w_d).
\end{equation}
for all $\bv=v_1\dots v_d, \bw=w_1\cdots w_d \in V^d$. 
Since the form is even, we have for any $\si\in\Si_d$:
\begin{equation}\label{EFormSign}
((v_1 \otimes \cdots \otimes v_d)^\si, (w_1 \otimes \cdots \otimes w_d)^\si)_d = (v_1 \otimes \cdots \otimes v_d, w_1 \otimes \cdots \otimes w_d)_d.
\end{equation}
Moreover, for $\bb,\bb'\in (B^V)^d$, we have 
\begin{equation}\label{EFormBasis}
(b_1' \otimes \dots \otimes b_d', b^*_1 \otimes \dots \otimes b^*_d)_d =(-1)^{\lan\bb',\bb\ran}\de_{\bb,\bb'}. 
\end{equation}

\begin{Lemma}\label{TBilY}
Let $\bb, \bb' \in \Se (B^V, d)$. Then $\bb^*:=b_1^*\cdots b_d^*\in \Se((B^{V})^*,d)$, and 
$(y_{\bb'}, y_{\bb^*})_d = \pm d!\, \delta_{\bb \sim \bb'} $. 
\end{Lemma}
\begin{proof}
By (\ref{EFormBasis}), $(y_\bb', y_{\bb^*})_d\neq 0$ only if $\bb \sim \bb'$. So we may assume that $\bb'=\bb$ and that the stabilizer $\Si_\bb=\Si_{\bb^*}$ is a standard parabolic subgroup. As 
no odd element repeats in $\bb$, we have
$$
	| \Si_\bb| = \bigg( \prod_{b\in B^V_\a} [\bb : b]! \bigg) \bigg( \prod_{b\in B^V_\c} [\bb: b]! \bigg)
	= \bigg( \prod_{b^*\in (B^{V})^*_\c} [\bb^* : b^*]! \bigg) \bigg( \prod_{b\in B^V_\c} [\bb: b]! \bigg).
$$
So, using (\ref{EFormSign}) and (\ref{EFormBasis}), we have that
$(y_\bb, y_{\bb^*})_d$ equals 
\begin{align*}
	\,& \bigg( \bigg( \prod_{b\in B^V_\c} [\bb: b]! \bigg) x_\bb , 
	\bigg( \prod_{b^*\in (B^{V})^*_\c} [\bb^* : b^*]! \bigg)
	x_{\bb^*}\bigg)_d\\
	=\,& \bigg( \bigg( \prod_{b\in B^V_\c} [\bb: b]! \bigg) \sum_{\si \in \Si_d / \Si_{\bb}} (b_1 \otimes \cdots \otimes b_d)^\si , 
	\bigg( \prod_{b^*\in (B^{V})^*_\c} [\bb^* : b^*]! \bigg)
	\sum_{\si \in \Si_d / \Si_{\bb}} (b_1^* \otimes \cdots \otimes b_d^*)^\si \bigg)_d\\
	=\,& \bigg( \prod_{b\in B^V_\c} [\bb: b]! \bigg) \bigg( \prod_{b^*\in (B^{V})^*_\c} [\bb^* : b^*]! \bigg) [\Si_d : \Si_\bb] (b_1 \otimes \cdots \otimes b_d, b_1^* \otimes \cdots \otimes b_d^*)_d\\
	=\,& \pm d!
\end{align*}
which completes the proof.
\end{proof}

In view of the lemma, we have $(z,w)_d$ is divisible by $d!$ for all $z,w\in \tilde\Ga^d V_R$. So we can define a new form on $\tilde\Ga^d V_R$ by setting 
\begin{equation}\label{ESimForm}
(z,w)_\sim:=\frac{1}{d!}(z,w)_d
\end{equation}
for all $z,w\in \tilde\Ga^d V_R$. The following is now clear from the lemma:

\begin{Proposition}\label{GaTiBil}
The bilinear form $( \cdot, \cdot)_\sim$ on $\tilde \Ga^d V_R$ is even and non-degenerate. Moreover, it is supersymmetric or superantisymmetric depending on the parity of $d$ and whether $(\cdot, \cdot)$ is supersymmetric or superantisymmetric. 
\end{Proposition}

\subsection{Modified divided powers of algebras}
\label{SSGaA}
Let $A_R=A_{R,\0}\oplus A_{R,\1}$ be an $R$-superalgebra such that $A_{R,\0}$ and $A_{R,\1}$ are free of finite rank as $R$-modules. We say that $A_R$ is a {\em calibrated superalgebra} if we are given a free $R$-module decomposition $A_{R,\0}=\a_R\oplus \c_R$ such that $\a_R$ is a unital subalgebra of $A_R$. Choose bases $B^A_\a,B^A_\c,B^A_\1$ of $\a_R,\c_R,A_{R,\1}$, respectively, and set $B^A=B^A_\a\cup B^A_\c\cup B^A_\1$. Our main examples come from based quasihereditary algebras over $R$ with conforming heredity data as in \ref{SSBQHA}.  In that case we would take $\a_R$ to be the $R$-span of $B_\a$, cf. (\ref{EBA}). 

Note that a calibrated $R$-superalgebra is in particular a calibrated $R$-supermodule as in \S\ref{SSCalMod}, so we have modified divided power $\tilde\Ga^dA_R$ and elements $x_\bb\in\Ga^dA_R$, $y_\bb\in\tilde\Ga^dA_R$ for $\bb\in\Se(B^A,d)$. But to differentiate between algebras and module, when working with algebras, it will be convenient to use another notation:
\begin{equation}\label{EOldEta}
\xi^\bb:=x_\bb\quad\text{and}\quad \eta^\bb:=y_\bb.
\end{equation}
Recall from \S\ref{SSSpaces} that $\Ga^dA_R$ is a subsuperalgebra of $A_R^{\otimes d}$. Moreover, $\tilde\Ga^d A_R$ is a (unital) subsuperalgebra of $\Ga^dA_R$, see \cite[Proposition~ 3.12]{KMgreen2}, with basis 
\begin{equation}\label{EEtaBasisOne}
\{\eta^\bb\mid \bb\in\Se(B^V,d)/\Si_d\}.
\end{equation}

Recall from \cite[\S4.1]{EK1}, that \(\bigoplus_{d\geq 0} A_R^{\otimes d}\) is a bisuperalgebra 
with 
the coproduct \(\coproduct\) defined so that
\begin{align*}
	\coproduct\,:\,\, A_R^{\otimes d}\,\, &\to \,\,\bigoplus_{c=0}^d A_R^{\otimes c} \otimes A_R^{\otimes (d-c)},\\
	a_1 \otimes \cdots \otimes a_d\,\,&\mapsto\,\, \sum_{c=0}^d (a_1 \otimes \cdots \otimes a_c) \otimes (a_{c+1} \otimes \cdots \otimes a_d).
\end{align*}
Moreover, by the $n=1$ case of \cite[Corollary 3.24]{KMgreen2}, 
$\bigoplus_{d\geq 0} \tilde\Ga^d A_R$
is a sub-bisuperalgebra of\, $\bigoplus_{d\geq 0} A_R^{\otimes d}$.

There is also another bisuperalgebra structure on $\bigoplus_{d\geq 0} A_R^{\otimes d}$ using the star product $*$ of  (\ref{EStarNotationGen}). In fact:

\begin{Lemma} \label{LNablaStarn=1} {\rm \cite[Corollary 4.4]{KMgreen2}} \ 
$\bigoplus_{d\geq 0} \tilde\Ga^d A_R$ is a sub-bisuperalgebra of $\bigoplus_{d\geq 0} A_R^{\otimes d}$ with respect to the coproduct $\coproduct$ and the product $*$
\end{Lemma}

\subsection{\boldmath $\tilde\Ga^d V$ as a module over $\tilde\Ga^d A$}
\label{SSGammaTildeMod}
Let $A_R=\a_R\oplus \c_R\oplus A_{R,\1}$ be a calibrated superalgebra as in the previous subsection. 
Let $V_R=V_{R,\a}\oplus V_{R,\c}\oplus V_{R,\1}$ be a calibrated $R$-supermodule with the corresponding basis $B^V=B^V_\a\sqcup B^V_\c\sqcup B^V_\1$ as in \S\ref{SSCalMod}, and assume in addition that $V_R$ is an $A_R$-supermodule. We say that $V_R$ is a {\em calibrated} $A_R$-supermodule if $\a_R V_{R,\a}\subseteq V_{R,\a}$.

Recall from \S\ref{SSSpaces} that $\Gamma^d V_R$ is naturally a $\Gamma^d A_R$-supermodule. So upon restriction to the subalgebra 
$\tilde \Ga^d A_R\subseteq \Ga^d A_R$,   
$\Ga^d V_R$ is  
a $\tilde\Ga^d A_R$-supermodule. In Lemma~\ref{GaTiMod} we will show that $\tilde\Ga^dV_R\subseteq \Ga^dV_R$ is a $\tilde\Ga^dA_R$-subsupermodule if $V_R$ is a calibrated $A_R$-supermodule. .  

For $a\in A_R$ and $b,c\in B^V$, we define the structure constants $\ka^b_{a,c}\in R$ from $ac=\sum_{b\in B^V}\ka^b_{a,c} b$.
For $\ba\in \Se(B^A,d)$ and $\bb,\bc\in \Se(B^V,d)$, we also set 
$\kappa_{\ba, \bc}^{\bb} := \kappa_{a_1, c_1}^{b_1} \cdots \kappa_{a_d, c_d}^{b_d}$. 
We want to describe the structure constants $f^{\bb}_{\ba, \bc}$ defined from 
$$\xi^\ba x_{\bc} 
= \sum_{\bb \in \Se (B^V, d)/\Si_d} f^{\bb}_{\ba, \bc} x_{\bb}.$$ 
Recall the stabilizer $\Si_\bs$ from \S\ref{SGeneralNot}. 
The following lemma is an analogue of \cite[Corollary~3.7]{KMgreen2}: 

\begin{Lemma}\label{GaPr}
Let $\ba \in \Se (B^A, d)$ and $\bb,\bc \in \Se (B^V, d)$. 
Let $X$ be the set of all pairs $(\ba',\bc')\in \Se (B^A, d)\times\Se (B^V, d)$ such that $\ba'\sim\ba$, $\bc'\sim\bc$ and $|a'_k|+|c_k'|=|b_k|$ for all $k=1,\dots,d$. 
Then  
\[ f^{\bb}_{\ba, \bc} = \sum_{(\ba', \bc') \in X/\Si_\bb} (-1)^{\lan \ba \ran + \lan \ba' \ran + \lan \bc \ran + \lan \bc' \ran + \lan \ba', \bc' \ran} [\Si_{\bb} : (\Si_{\bb} \cap \Si_{\ba'} \cap \Si_{\bc'})] \kappa_{\ba', \bc'}^{\bb} \]
\end{Lemma}
\begin{proof}
Clearly, we have 
$$
f^{\bb}_{\ba, \bc} =\sum
(-1)^{\lan \ba \ran + \lan \ba' \ran + \lan \bc \ran + \lan \bc' \ran + \lan \ba', \bc' \ran} \kappa_{\ba', \bc'}^{\bb},
$$
the sum being over all $(\ba',\bc')\in\Se(B^A,d)\times \Se(B^V,d)$ such that $\ba'\sim\ba$ and $\bc'\sim\bc$, cf. \cite[(3.14)]{EK1}. It remains to note that $\kappa_{\ba', \bc'}^{\bb}=0$ unless $(\ba', \bc') \in X$, and for $(\ba', \bc'),(\ba'', \bc'') \in X$ in the same $\Si_\bb$-orbit the corresponding summands are equal to each other, cf. the proof of \cite[Corollary~3.7]{KMgreen2}. 
\end{proof}

\begin{Lemma}\label{GaTiMod}
If $V_R$ is a calibrated $A_R$-supermodule, then $\tilde\Ga^dV_R\subseteq \Ga^dV_R$ is a $\tilde\Ga^dA_R$-submodule.
\end{Lemma}
\begin{proof}
For each $\ba \in \Se(B^A, d)$ and $\bc \in  \Se(B^V, d)$ we have (working over the field of quotients of $R$):
\begin{align*}
	\eta^{\ba} y_{\bc} = ([\ba]_\c^! \xi^\ba )( [\bc]_\c^! x_\bc) = 
	\sum_{\bb \in \Se (B^V, d)/\Si_d} [\ba]_\c^! [\bc]_\c^! f^\bb_{\ba,\bc} x_{\bb}
	= 
	\sum_{\bb \in \Se (B^V, d)/\Si_d} \frac{[\ba]_\c^! [\bc]_\c^!f^\bb_{\ba,\bc}}{[\bb]_\c^!} y_{\bb}. 
\end{align*}
So in view of Lemma \ref{GaPr}, it suffices to prove that for fixed 
$\ba \in \Se(B^A, d)$, $\bb,\bc \in  \Se(B^V, d)$ and $(\ba',\bc')\in X$ satisfying $\ka^{\bb}_{\ba',\bc'}\neq 0$, the integer 
$$M_{\ba, \bc}^{\bb} := [\ba]_\c^! [\bc]_\c^! [\Si_{\bb} : (\Si_{\bb} \cap \Si_{\ba'} \cap \Si_{\bc'})]
$$ is divisible by $[\bb]_\c^!$. Here, as in Lemma~\ref{GaPr},  $X$ consists of all pairs  $(\ba',\bc')\in \Se(B^A,d)\times \Se(B^V, d)$ such that $\ba'\sim\ba$, $\bc'\sim\bc$, and $|a_k'|+|c_k'|=|b_k|$ for all $k=1,\dots,d$. 

For $a\in B^A$ and $b,c\in B^V$, let
\[ m_{a,c}^b = \# \{ k \in [d] \mid a'_k = a, c'_k = c, b_k = b\}. \]
Then, recalling the notation (\ref{EMultiplicity}),  we have 
\begin{align*}
	| \Si_{\bb} \cap \Si_{\ba'} \cap \Si_{\bc'} | &= \prod_{a\in B^A, b,c \in B^V} m_{a, c}^b !,
\\
	[\ba : a ] &=[\ba' : a ]= \sum_{b, c \in B^V} m_{a, c}^b, \\ 
	[\bc : c ]&=[\bc' : c ] = \sum_{a \in B^A, b\in B^V} m_{a, c}^b,\\ 
	[\bb : b ] &= \sum_{a \in B^A, c \in B^V} m_{a, c}^b.
\end{align*}
In particular, for all $b,c\in B^V$ and $a\in B^A$, we have integers 
\[ z_b := \frac{ [\bb : b ]!}{\prod_{a \in B^A, c \in B^V} m_{a, c}^b! },
\ 
Z_c:=\frac{ [\bc : c]!}{\prod_{a \in B^A_\a, b \in B_\c^V} m_{a,c}^b!},\ 
Z_a:=\frac{ [\ba: a]!}{\prod_{b \in B_\c^V, c \in B_\0^V}m_{a,c}^b!}. 
\]
Denoting $C = \prod_{b \in B_\a^V \sqcup B_\1^V} z_b$, we have
\[ [\Si_\bb : \Si_\bb \cap \Si_\ba' \cap \Si_\bc' ] = \frac{\prod_{b \in B^V} [\bb :b]!}{\prod_{a\in B^A, b,c \in B^V} m^b_{a,c}!}= \prod_{b \in B^V} z_b = C \prod_{b \in B_\c^V} z_b. \]

Let $b \in B_{\c}^V$.
If $a \in B_{\1}^A$ or $c \in B^V_\1$, then $m_{a,c}^b \leq 1$  because there are no repeated odd elements in tuples in $\Se (B^A, d)$ or $\Se(B^V, d)$. Also observe that if $a \in B_\a^A$ and $c \in B_\a^V$, then $ac \in V_{R,\a}$ by assumption, so, since $b \in B_{\c}^V$, we have $\ka^b_{a,c}= 0$, hence $m_{a,c}^b = 0$. So 
\[ z_b = \frac{ [\bb : b]!}{ \left( \prod_{a \in B^A_{\c}, c \in B_{\0}^V} m_{a,c}^b! \right) \left( \prod_{a \in B^A_{\a}, c \in B_{\c}^V} m_{a,c}^b! \right) }. \]

Thus we have 
\begin{align*}
	M_{\ba, \bc}^\bb &= \Bigg( \prod_{a \in B_\c^A} [\ba : a]! \Bigg)  \Bigg( \prod_{c \in B_\c^V} [\bc : c]! \Bigg)
	\cdot C  
	\prod_{b \in B_\c^V} \frac{[\bb : b]!}{\left( \prod_{a \in B^A_{\c}, c \in B_{\0}^V} m_{a,c}^b! \right) \left( \prod_{a \in B^A_{\a}, c \in B_{\c}^V} m_{a,c}^b! \right) }\\
	&= C\Bigg( \prod_{a \in B^A_\c} \frac{ [\ba: a]!}{\prod_{b \in B_\c^V, c \in B_\0^V} m_{a,c}^b!} \Bigg) 
	\Bigg( \prod_{c \in B_\c^V} \frac{ [\bc : c]!}{\prod_{a \in B^A_\a, b \in B_\c^V} m_{a,c}^b!} \Bigg)
	\Bigg( \prod_{b \in B^V_\c} [\bb: b]! \Bigg)\\
	&=  C\Bigg( \prod_{a \in B^A_\c} Z_a \Bigg) 
	\Bigg( \prod_{c \in B_\c^V} Z_c  \Bigg)
	[\bb]_\c^!,
\end{align*}
which completes the proof.
\end{proof}

\subsection{\boldmath More on $\tilde\Ga^d A_R$-module $\tilde\Ga^d V_R$}

Throughout the subsection $A_R$ is a calibrated superalgebra and $V_R$ is a calibrated $A_R$-supermodule. 

Let $e\in\a_R$ be an idempotent such that $be=b$ or $0$ for all $b\in B^A$, cf. \cite[\S5]{KMgreen2}. 
Let $B^Ae=\{b\in B^A\mid be=b\}$, $B^A_\a e:=\{b\in B^A_\a\mid be=b\}$ and $B^A_\c e:=\{b\in B^A_\c\mid be=b\}$. 
We have an idempotent $\eta^{e^d}=e^{\otimes d}\in\tilde\Ga^d A_R$. In the special case where $V_R=A_Re$, we 
always take $V_{R,\a}:=\a_R e$ with basis $B^A_\a e$ and $V_{R,\c}:=\c_R e$ with basis $B^A_\c$. In this case we 
can describe the $\tilde\Ga^dA_R$-module $\tilde\Ga^dV_R$ explicitly as follows: 

\begin{Lemma}\label{GaTiIdem}
Let $e\in\a_R$ be an idempotent such that $be=b$ or $0$ for all $b\in B^A$. Then $\tilde \Ga^d (A_Re) \cong (\tilde\Ga^dA_R) \eta^{e^d}$.
\end{Lemma}
\begin{proof}
Note that $\tilde \Ga^d (A_Re)$ has basis $\{y_\bb\mid \bb\in\Se(B^Ae,d)/\Si_d\}$ and $(\tilde\Ga^dA_R) \eta^{e^d}$ has basis $\{\eta^\bb\mid \bb\in\Se(B^Ae,d)/\Si_d\}$. There is a $\tilde\Ga^dA_R$-module map $\phi:(\tilde\Ga^dA_R) \eta^{e^d}\to \tilde \Ga^d (A_Re)$ with $\phi(\eta^{e^d})=y_{e^d}$. It remains to notice that $\phi(\eta^\bb)= y_\bb$ for all $\bb\in\Se(B^Ae,d)/\Si_d$. 
\end{proof}

Suppose there is an even superalgebra anti-involution $\tau : A_R \to A_R$. 
We make the additional assumption that $\tau(\a)=\a$, in which case $\tau^{\otimes d}$   restricts to a superalgebra anti-involution $\tau_d$ on $\tilde \Ga^d A_R$, see \cite[(4.12)]{KMgreen2}. 

Given $W\in \mod{\tilde \Ga^d A_R}$, its $\tau_d$-dual $W^{\tau_d}$ is defined as 
$W^*:=\Hom_R(W,R)$ with the action $(xf)(w)=(-1)^{|x||f|}f(\tau_d(x)w)$ for all $f\in W^*, w\in W,x\in \tilde \Ga^d A_R$. Note that $W\simeq W^{\tau_d}$ if and only if there is a non-degenerate $\tau_d$-contravariant form $(\cdot,\cdot)$ on $W$, where $\tau_d$-contravariance means 
$(xv, w) = (-1)^{|x||v|}(v, \tau_d(x)w)$ for all $x\in  \tilde \Ga^d A$ and $v,w\in W$. 

\begin{Lemma}\label{GaTiForm}
Let $(\cdot,\cdot)$ be an even supersymmetric or superantisymmetric, non-degenerate bilinear form on $V_R$,
such that $(V_{R,\a},V_{R,\a})=0$ and the $R$-complement $V_{R,\c}$ of $V_{R,\a}$ in $V_{R,\0}$ can be chosen so that the restriction of $(\cdot,\cdot)$ to $V_{R,\a}\times V_{R,\c}$ is a perfect pairing. Then $(\cdot,\cdot)_\sim$ is a non-degenerate $\tau_{d}$-contravariant form on the $\tilde\Ga^d A_R$-module $\tilde\Ga^d V_R$. 
\end{Lemma}
\begin{proof}
For $\ba=a_1\cdots a_d \in A^d$ and $\bv=v_1\cdots v_d\in V^d$, we set $\ba \cdot \bv:=(a_1 v_1) \cdots (a_d v_d)\in V^d$. Then 
$
(-1)^{\lan \ba \cdot \bv, \bw \ran} = (-1)^{\lan \ba, \bw \ran + \lan \bv, \bw \ran}.
$ 
Using this, it is easy to establish that $(\cdot, \cdot)_d$ is a $\tau^{\otimes d}$-contravariant form on $V^{\otimes d}$, cf. (\ref{TBil}). The lemma now follows from Proposition~\ref{GaTiBil}. 
\end{proof}

Recall the coproduct from from \S\ref{SSGaA}, for supermodules $X\in\mod{\tilde\Ga^d A_R}$ and $Y\in\mod{\tilde\Ga^e A_R}$, we have a structure of a $\tilde\Ga^{d+e} A_R$-module on $X\otimes Y$. 
Recalling 
Lemma~\ref{LNablaStarn=1} and 
the isomorphism from Lemma~\ref{GaTiSumOld}, we now obtain:

\begin{Lemma} \label{GaTiSum}
Let $V_R$ and $W_R$ be calibrated $A_R$-supermodules. Then we have an isomorphism of $\tilde \Ga^d A_R$-modules 
\[ \bigoplus_{d_1 + d_2 = d} (\tilde \Ga^{d_1} V_R)  \otimes (\tilde \Ga^{d_2} W_R)\iso \tilde \Ga^d (V_R \oplus W_R),\ y\otimes y'\mapsto y*y' . \]
\end{Lemma}

\section{Generalized Schur algebras}
\label{SSA}
Throughout the section we fix $n\in\Z_{> 0}$ and a calibrated  superalgebra $A_R$ over $R$ as in \S\ref{SSGaA}. In particular we have fixed $\a_R,\c_R,A_{R,\1}$ with bases $B_\a,B_\c,B_\1$, and $B=B_\a\sqcup B_\c\sqcup B_\1$ is a basis of $A_R$. 

\subsection{\boldmath The algebra $T^A(n,d)$}
\label{SSDefT}
We consider the matrix superalgebra $M_n(A_R)$ with $M_n(A_R)_\0=M_n(A_{R,\0})$ and $M_n(A_R)_\1=M_n(A_{R,\1})$. 
For $b\in B$, we denote by  
$
\xi_{r,s}^b\in M_n(A_R)
$ 
so that 
\begin{equation}\label{EMBasis}
B^{M_n(A)}:=\{\xi_{r,s}^b\mid b\in B,\, 1\leq r,s\leq n\}
\end{equation}
 is a basis of $M_n(A_R)$. 

This superalgebra is calibrated via
$M_n(A_R)_\a:=M_n(\a_R)$ and $M_n(A_R)_\c:=M_n(\c_R)$. So 
for $d\in \Z_{\geq 0}$, we may define
$$
T^{A}(n,d)_R:=\tilde\Ga^d M_n(A_R).
$$ 
The theory developed in Section~\ref{SDiv} applies with $A_R$ replaced by $M_n(A_R)$. To facilitate the transition from $A_R$ to $M_n(A_R)$, it is convenient to adopt a special notation. First of all, taking into account that the basis of $M_n(A_R)$ is of the form (\ref{EMBasis}), we have element $\eta^\bx$ from (\ref{EOldEta}) labeled by $d$-tuples $\bx=\xi^{b_1}_{r_1,s_1}\,\dots\,\xi^{b_d}_{r_d,s_d}$ of basis elements from (\ref{EMBasis}) such that $(b_i,r_i,s_i)=(b_j,r_j,s_j)$ for $i\neq j$ only if $b_i$ is even. It is convenient to denote such $\eta^\bx$ rather by $\eta^\bb_{\br,\bs}$ where $\bb:=b_1\dots b_d$, $\br:=r_1\cdots r_d$ and $\bs:=s_1\dots s_d$. Thus, 
\begin{equation*}\label{EXiDef}
	\eta_{\br,\bs}^\bb:= [\bb,\br,  \bs]^!_{\c}\sum_{(\bc,\bt,\bu)\sim(\bb,\br,\bs)} 
	(-1)^{\lan\bb,\br,\bs\ran+\lan\bc,\bt,\bu\ran}
	\xi_{t_1,u_1}^{c_1}\otimes\dots\otimes \xi_{t_d,u_d}^{c_d}
	\in T^A(n,d)_R,
\end{equation*}
where
\begin{align}
	\label{EAngle3}
	\lan\bb, \br,  \bs\ran
	&:=|\{(k,l)\in[d]^2\mid k<l,\ b_k,b_l\in B_\1,\ (b_k,r_k,s_k)> (b_l,r_l, s_l)\}|
\end{align}
is the analogue of (\ref{EBBSign}) (where `$>$' is a fixed total order on $B\times [n]\times [n]$) and 
\begin{equation}\label{ECFactorial}
	[\bb,\br, \bs]^!_{\c} :=\prod_{ b\in B_\c,\, r,s\in [n]}|\{k\in[d]\mid  (b_k,r_k,s_k)=(b,r,s)\}|!
\end{equation}
is the analogue of (\ref{EBBFactorial}). 

Define $\Seq^B (n,d)$ to be the set of all triples 
$$
(\bb,\br, \bs) = ( b_1\cdots b_d,\, r_1\cdots r_d,\, s_1\cdots s_d ) \in  B^d\times[n]^d\times[n]^d
$$
such that for all $1\leq k\neq l\leq d$ we have 
$(b_k,r_k,s_k)=(b_l,r_l, s_l)$ 
only if $b_k\in B_\0$. To connect with the theory developed in \S\ref{SDiv}, we identify $\Seq^B (n,d)$ with $\Se(B^{M_n(A)},d)$ via the map 
$$
\Seq^B (n,d)\iso\Se(B^{M_n(A)},d),\ (\bb,\br, \bs)\mapsto \xi^{b_1}_{r_1,s_1}\,\cdots\,\xi^{b_d}_{r_d,s_d}.
$$ 
Since $\Seq^B (n,d)\subseteq B^d\times[n]^d\times[n]^d$ is a $\Si_d$-invariant subset, we can choose a corresponding set $\Seq^B (n,d)/\Si_d$ of $\Si_d$-orbit representatives and identify it with the set of all $\Si_d$-orbits on $\Seq^B (n,d)$ as in \S\ref{SGeneralNot}. Then we have a basis
$$
\{\eta_{\br,\bs}^\bb\mid (\bb,\br,\bs)\in \Seq^B (n,d)/\Si_d\}
$$
of $T^A(n,d)_R$ which is the analogue of the basis (\ref{EEtaBasisOne}). 

Applying the theory developed in \S\ref{SSGaA} replacing $A_R$ with $M_n(A_R)$, we also have a coproduct $\coproduct$ on $\bigoplus_{d\geq 0} T^A(n,d)_R$, which allows us to consider the  tensor product $V\otimes W$ of $V\in\mod{T^A(n,d)_R}$ and $W\in\mod{T^A(n,e)_R}$ as a supermodule over $T^A(n,d+e)_R$. 





Extending scalars from $R$ to $\F$, we now define the $\F$-superalgebra 
$$T^A(n,d):=\F\otimes_R T^A(n,d)_R.$$
We denote $1_\F\otimes \eta^\bb_{\br,\bs}\in T^A(n,d)$ again by $\eta^\bb_{\br,\bs}$, the map $\id_\F\otimes\, \coproduct$ again by $\coproduct$, etc. In particular, given 
$V\in\mod{T^A(n,d)}$ and $W\in\mod{T^A(n,e)}$, we consider $V\otimes W$ as a $T^A(n,d+e)$-supermodule via $\coproduct$.

\subsection{\boldmath Quasi-hereditary structure on $T^A(n,d)$}
\label{SSQHT}
Throughout the section, we assume that $d\leq n$.

Let $A_R$ be a based quasi-hereditary superalgebra with conforming heredity data $I,X,Y$, see \S\ref{SSBQHA}. In particular, $A_R$ comes with the standard idempotents $\{e_i\mid i\in I\}$ in $\a_R$.

For $\la\in\La(n,d)$, set $
\bl^\la:=1^{\la_1}\cdots n^{\la_n}\in[n]^d.
$ 
For an idempotent $e\in \a_R$ we have an idempotent 
$
\eta_\la^e:=\eta^{e^d}_{\bl^\la,\bl^\la}\in T^A(n,d)$. For all \(\bla=(\la^{(0)},\dots,\la^{(\ell)}) \in \La^I(n,d)\), we have the orthogonal idempotents 
$$
\eta_{\bla}:=\eta_{\la^{(0)}}^{e_0} * \cdots * \eta_{\la^{(\ell)}}^{e_\ell} \in T^{A}(n,d).
$$
Given a $T^A(n,d)$-supermodule $V$ we refer to the vectors of $\eta_\bla V$ as vectors of {\em weight $\bla$}. 

For $\la=(\la_1,\dots,\la_n)\in\La(n)$, define the monomial  $z^\la:=z_1^{\la_1}\cdots z_n^{\la_n}\in\Z[z_1,\dots,z_n].$ 
For $\bla\in\La^I(n)$, we now set 
$$z^\bla:=z^{\la^{(0)}}\otimes z^{\la^{(1)}}\otimes\dots\otimes z^{\la^{(\ell)}}\in \Z[z_1,\dots,z_n]^{\otimes I}.$$
Following \cite[\S5A]{KMgreen2}, see especially \cite[Lemma~5.9]{KMgreen2}, for a $T^A(n,d)$-supermodule $V$, we define its {\em formal character}  
$$
\ch V := \sum_{\bmu\in\La^I(n,d)}(\dim\, \eta_\bmu V)z^\bmu\in \Z[z_1,\dots,z_n]^{\otimes I}.
$$
If $\sum_{i\in I}e_i=1_A$, then 
$1_{T^{A}(n,d)}=\sum_{\bla\in\La^{I}(n,d)}\eta_\bla,
$
but we do not need to assume this. So in general we might have $\sum_{\bmu\in\La^I(n,d)}\eta_\bmu V\subsetneq V$ for $V\in\mod{T^A(n,d)}$. (Another fact that is not going to be used directly is that $\ch V$ is a symmetric function.)

\begin{Lemma}\label{LChProd} \cite[Lemma 5.10]{KMgreen2}
If\, $V\in\mod{T^A(n,d)}$ and\, $W\in\mod{T^A(n,e)}$, then\, 
$
\ch(V \otimes W) = \ch(V) \, \ch(W).
$
\end{Lemma}

By \cite[Theorem 6.6]{KMgreen3}, $T^A(n,d)$ is a based quasi-hereditary algebra with standard idempotents $\{\eta_\bla\mid\bla\in \La_+^I(n,d)\}$,  see \cite[\S6]{KMgreen3}. The corresponding simple, standard and costandard modules are denoted
$$
\{L(\bla)\mid\bla\in \La_+^I(n,d)\},\quad \{\Delta(\bla)\mid\bla\in \La_+^I(n,d)\},\quad \{\nabla(\bla)\mid\bla\in \La_+^I(n,d)\}, 
$$
respectively.

Recalling the $X$-colored standard tableaux from \S\ref{SSComb}, for $\bla\in\La_+^I(n,d)$, the standard module $\De(\bla)$ has basis 
\begin{equation}\label{EBasisDe}
\{v_\T 
\mid \T\in\Std^X(\bla)\}.
\end{equation} 
Moreover, if $\T\in \Std^X(\bla,\bmu)$ for some $\bmu\in\La^I(n,d)$, see (\ref{EAl}), then 
\begin{equation}\label{EWtvT}
v_\T\in\eta_{\bmu}\De(\bla).
\end{equation}

\begin{Lemma} \label{LChInd} {\rm \cite[Lemma 3.23]{KW1}} The formal characters $\{\ch\De(\bla)\mid \bla\in\La_+^I(n,d)\}$ are linearly independent.
\end{Lemma}

Recalling the notation (\ref{EIota}), we have:

\begin{Lemma} \label{LDeCol} {\rm \cite[Theorem 6.17(i)]{KMgreen3}}
	Let $\bla=(\la^{(i)})_{i\in I}\in\La_+^I(n,d)$. Then $$\De(\bla)\simeq\bigotimes_{i\in I}\De(\biota_i(\la^{(i)})).$$ 
\end{Lemma}

For $\bla=(\la^{(j)})_{j\in I}\in\La^I_+(n,d)$, $\bmu=(\mu^{(j)})_{j\in I}\in\La^I_+(n,e)$ and $\bnu=(\nu^{(j)})_{j\in I}\in\La_+^I(n,d+e)$ we define
\begin{equation}\label{EBoldLR}
c^{\,\bnu}_{\bla,\bmu}:=\prod_{j\in I} c^{\,\nu^{(j)}}_{\la^{(j)}\hspace{-.6mm},\mu^{(j)}}.
\end{equation}

We will crucially use the following

\begin{Theorem}\label{DeChTens}{\rm \cite[Main Theorem, Corollary 3.30]{KW1}} 
Let $\bla\in\La^I_+(n,d)$ and $\bmu\in\La^I_+(n,e)$. If $d+e\leq n$ then $\De(\bla) \otimes \De(\bmu)$ (resp. $\nabla(\bla) \otimes \nabla(\bmu)$) has a standard (resp. costandard) filtration, and
$$
\ch(\De(\bla)\otimes\De(\bmu))=\sum_{\bnu\in\La_+^I(n,d+e)}c^{\,\bnu}_{\bla,\bmu}\ch\De(\bnu).
$$
\end{Theorem}

An immediate corollary of Theorem~\ref{DeChTens} is

\begin{Corollary}\label{CTensTilt}
Let $\Tilt\in\mod{T^A(n,d)}$ and $ \Tilt'\in\mod{T^A(n,e)}$ be tilting supermodules. If $d+e\leq n$ then the tensor product  $\Tilt \otimes \Tilt'$ is a tilting supermodule over $T^A(n, d+e)$.
\end{Corollary}

\subsection{\boldmath Modified divided powers of modules over $M_n(A_R)$}
\label{SSMNA}
In this subsection we work over $R$.

Let $A_R=\a_R\oplus \c_R\oplus A_{R,\1}$ be a calibrated superalgebra as in  \S\ref{SSGaA} with basis $B^A=B^A_\a\cup B^A_\c\cup B^A_\1$, and  $V_R=V_{R,\a}\oplus V_{R,\c}\oplus V_{R,\1}$ be a calibrated $A_R$-supermodule as in \S\ref{SSGammaTildeMod}. Recall that by definition $T^A(n,d)_R$ is $\tilde \Ga^d M_n(A_R)$ for the calibration on $M_n(A_R)$ induced by that on $A_R$.

The $R$-supermodule of column vectors  $\Co_n(V_R)=V_R^{\oplus n}$ is a left supermodule over $M_n(A_R)$ in a natural way. In fact, it is a calibrated $M_n(A_R)$-supermodule with $\Co_n(V_R)_\a:=\Co_n(V_{R,\a})$ and $\Co_n(V_R)_\c:=\Co_n(V_{R,\c})$. Then by Lemma~\ref{GaTiMod}, we have $\tilde\Ga^d\Co_n(V_R)$ is a left supermodule over $T^A(n,d)_R=\tilde \Ga^d M_n(A_R)$.

Extending scalars to $\F$ we get the left module 
$$\tilde\Ga^d\Co_n(V):=\F\otimes_R\tilde\Ga^d\Co_n(V_R)
$$ 
over $T^A(n,d)=\F\otimes_R T^A(n,d)_R$. 


\begin{Lemma}\label{GaCnIdem}
Let $e\in\a_R$ be an idempotent such that $be=b$ or $0$ for all $b\in B^A$.  Then $\tilde\Ga^d (\Co_n(A_Re)) \simeq T^A(n,d)_R \eta_{1^d, 1^d}^{e^d}$.
\end{Lemma}
\begin{proof}
First notice that $\Co_n(A_Re)\simeq M_n(A_R)\xi_{1,1}^e$ as $M_n(A_R)$-supermodules. The result now follows by applying Lemma~\ref{GaTiIdem}.
\end{proof}

We will also need the right module versions of the results of this subsection. If $V_R$ is a calibrated right $A_R$-supermodule, we consider the row vectors $\Ro_n(V_R)=V_R^{\oplus n}$ as a calibrated right $M_n(A_R)$-supermodule. Then we obtain a right $T^A(n,d)_R$-supermodule structure on $\tilde\Ga^d\Ro_n(V_R)$ and a right $T^A(n,d)$-module $\tilde\Ga^d\Ro_n(V):=\F\otimes_R\tilde\Ga^d\Ro_n(V_R)$. The right module analogue of Lemma~\ref{GaCnIdem} is then clear.

\section{\boldmath The extended zigzag Schur algebra}
\label{SZig}

\subsection{The extended zigzag algebra}
\label{SSZig}
Fix $\ell\geq 1$ and set 
$$
I:=\{0,1,\dots,\ell\},\quad
J:=\{0,1,\dots,\ell-1\}.
$$ 
Let $\Gamma$ be the quiver with vertex set $I$ and arrows $\{\za_{j,j+1},\za_{j+1,j}\mid j\in J\}$ as in the picture: 
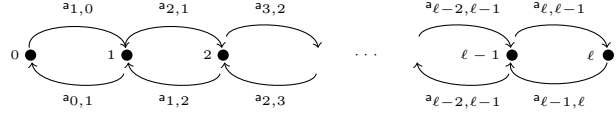
\begin{align*}
\begin{braid}\tikzset{baseline=3mm}
	\coordinate (0) at (-4,0);
	\coordinate (1) at (0,0);
	\coordinate (2) at (4,0);
	\coordinate (3) at (8,0);
	\coordinate (6) at (12,0);
	\coordinate (L1) at (16,0);
	\coordinate (L) at (20,0);
	\draw [thin, black,->,shorten <= 0.1cm, shorten >= 0.1cm]   (0) to[distance=1.5cm,out=100, in=100] (1);
	\draw [thin,black,->,shorten <= 0.25cm, shorten >= 0.1cm]   (1) to[distance=1.5cm,out=-100, in=-80] (0);
	\draw [thin, black,->,shorten <= 0.1cm, shorten >= 0.1cm]   (1) to[distance=1.5cm,out=100, in=100] (2);
	\draw [thin,black,->,shorten <= 0.25cm, shorten >= 0.1cm]   (2) to[distance=1.5cm,out=-100, in=-80] (1);
	\draw [thin,black,->,shorten <= 0.25cm, shorten >= 0.1cm]   (2) to[distance=1.5cm,out=80, in=100] (3);
	\draw [thin,black,->,shorten <= 0.25cm, shorten >= 0.1cm]   (3) to[distance=1.5cm,out=-100, in=-80] (2);
	\draw [thin,black,->,shorten <= 0.25cm, shorten >= 0.1cm]   (6) to[distance=1.5cm,out=80, in=100] (L1);
	\draw [thin,black,->,shorten <= 0.25cm, shorten >= 0.1cm]   (L1) to[distance=1.5cm,out=-100, in=-80] (6);
	\draw [thin,black,->,shorten <= 0.25cm, shorten >= 0.1cm]   (L1) to[distance=1.5cm,out=80, in=100] (L);
	\draw [thin,black,->,shorten <= 0.1cm, shorten >= 0.1cm]   (L) to[distance=1.5cm,out=-100, in=-100] (L1);
	\blackdot(-4,0);
	\blackdot(0,0);
	\blackdot(4,0);
	\blackdot(16,0);
	\blackdot(20,0);
	\draw(-4,0) node[left]{$0$};
	\draw(0,0) node[left]{$1$};
	\draw(4,0) node[left]{$2$};
	\draw(10,0) node {$\cdots$};
	\draw(13.4,0) node[right]{$\ell-1$};
	\draw(18.65,0) node[right]{$\ell$};
	\draw(-2,1.2) node[above]{$ \za_{1,0}$};
	\draw(2,1.2) node[above]{$ \za_{2,1}$};
	\draw(6,1.2) node[above]{$ \za_{3,2}$};
	\draw(14,1.2) node[above]{$ \za_{\ell-2,\ell-1}$};
	\draw(18,1.2) node[above]{$ \za_{\ell,\ell-1}$};
	\draw(-2,-1.2) node[below]{$ \za_{0,1}$};
	\draw(2,-1.2) node[below]{$ \za_{1,2}$};
	\draw(6,-1.2) node[below]{$ \za_{2,3}$};
	\draw(14,-1.2) node[below]{$ \za_{\ell-2,\ell-1}$};
	\draw(18,-1.2) node[below]{$ \za_{\ell-1,\ell}$};
\end{braid}
\end{align*}
The {\em extended zigzag algebra $\EZig$} is the path algebra $\k\Gamma$ modulo the following relations:
\begin{enumerate}
\item All paths of length three or greater are zero.
\item All paths of length two that are not cycles are zero.
\item All length-two cycles based at the same vertex are equivalent.
\item $ \za_{\ell,\ell-1} \za_{\ell-1,\ell}=0$.
\end{enumerate}
Length zero paths yield the standard idempotents $\{ \ze_i\mid i\in I\}$ with $ \ze_i  \za_{i,j} \ze_j= \za_{i,j}$ for all admissible $i,j$. The algebra $\EZig$ is graded by the path length: 
$\EZig=\EZig^0\oplus \EZig^1\oplus \EZig^2.
$ 
We consider $\EZig$ as a superalgebra with 
$\EZig_\0=\EZig^0\oplus \EZig^2\quad \text{and}\quad \EZig_\1=\EZig^1.
$ 
Define
$
\zc_j:= \za_{j,j+1} \za_{j+1,j} 
$ for all $j \in J$. 

The superalgebra $\EZig$ has an anti-involution $\tau$ with
$$
\tau( \ze_i)=  \ze_i,\  \tau(\za_{ij})=  \za_{ji},\  \tau(\zc_j) = -\zc_j.
$$

We consider the total order on $I$
given by $0<1<\dots<\ell$. For $i\in I$, we set 
\begin{equation}\label{EXY}
\zX(i):=
\left\{
\begin{array}{ll}
	\{\ze_i,\za_{i-1,i}\}   &\hbox{if $i>0$,}\\
	\{\ze_0\}  &\hbox{if $i=0$,}
\end{array}
\right.
\quad
\zY(i):=
\left\{
\begin{array}{ll}
	\{\ze_i,\za_{i,i-1}\}   &\hbox{if $i>0$,}\\
	\{\ze_0\}  &\hbox{if $i=0$.}
\end{array}
\right.
\end{equation}
With respect to this data we have:

\begin{Lemma} \label{LAQH} {\rm \cite[Lemma 4.14]{KMgreen1}}
The graded superalgebra $\EZig$ is based quasi-hereditary with conforming heredity data $I,\zX,\zY$. 
For the corresponding heredity basis $\zB$ we have 
$\zB_{\fa}=\{ \ze_i\mid i\in I\},\  \zB_{\fc} = \{ \zc_j\mid j\in J\},\ \zB_\1=\{ \za_{j,j+1}, \za_{j+1,j}\mid j\in J\}.
$
\end{Lemma}

For $i\in I$, let $\zL(i)=\k\cdot \zv_i$ with $|\zv_i|=\0$ and the action $\ze_i\zv_i=\zv_i$, $\zb\zv_i=0$ for all $\zb\in\zB\setminus\{\ze_i\}$. This makes $\zL(i)$ a $\EZig$-supermodule, and, up to isomorphism, $\{\zL(i)\mid i\in I\}$ is a complete set of irreducible $\EZig$-supermodules. Note that $L(i)^\tau\simeq L(i)$.

The standard modules $\De(i)$ have similarly explicit description: $\De(0)\simeq \zL(0)$, and for $i>0$,  $\De(i)$ has basis $\{\zv_i,\zw_{i}\}$ with $|\zv_i|=\0$, $|\zw_{i}|=\1$ and the only non-trivial action is: $\ze_i\zv_i=\zv_i$, $\ze_{i-1}\zw_i=\zw_i$, $\za_{i-1,i}\zv_i=\zw_i$. For the costandard modules we have 
$\nabla(0)\simeq\De(0)^\tau\simeq \zL(0)$, and for $i>0$,  $\nabla(i)\simeq\De(i)^\tau$ has basis $\{\zv_i^*,\zw_{i}^*\}$ with $|\zv_i^*|=\0$, $|\zw_{i}^*|=\1$ and the only non-trivial action is: $\ze_i\zv_i^*=\zv_i^*$, $\ze_{i-1}\zw_i^*=\zw_i^*$, $\za_{i,i-1}\zw_i^*=-\zv_i^*$. 

The indecomposable tilting supermodules over $\EZig$ are as follows: $\Tiltz(0)\simeq\zL(0)$ and $\Tiltz(i)=\Pi\EZig \ze_{i-1}$ for $i>0$. It will actually be more convenient for us to work with the tilting modules $\Pi\Tiltz(i)$. Thus $\Pi\Tiltz(0)=\Pi\zL(0)$ is $1$-dimensional with basis $\{\zv_0\}$ where $\zv_0$ is odd, $\Pi\Tiltz(1)=\EZig \ze_{0}$ has basis $\{\ze_0,\za_{1,0},\zc_0\}$, and, for $i>1$, $\Pi\Tiltz(i)=\EZig \ze_{i-1}$ has basis $\{\ze_{i-1},\za_{i-2,i-1},\za_{i,i-1},\zc_{i-1}\}$. For $i>0$, $\Pi\Tiltz(i)$ has a standard filtration $\Pi\De(i)|\De(i-1)$ (this means $\Pi\De(i)\subseteq \Tiltz(i)$ and $\Tiltz(i)/\Pi\De(i)\simeq\De(i-1)$) and a  costandard filtration $\nabla(i-1)|\Pi\nabla(i)$.

We have a full tilting module 
$$\Tiltz:=\bigoplus_{i\in I}\Pi\Tiltz(i)=\Pi L(0)\oplus \bigoplus_{i=0}^{\ell-1}\EZig\ze_i.$$ 
and the Ringel dual algebra $\EZig'=\End_\EZig(\Tiltz)^\sop$. In fact, $\EZig$ is Ringel self-dual, i.e. there is an isomorphism of superalgebras $\EZig'\cong\EZig$ which we now proceed to describe. 

For any $i\in I$, let $\iota_i:\Pi\Tiltz(i)\to \Tiltz$ and $\pi_i:\Tiltz\to \Pi\Tiltz(i)$ be the natural embedding and projection.
We have the right multiplication maps 
$\rho_{\zc_i}:\EZig\ze_i\to \EZig\ze_i, \zv\mapsto \zv\zc_i$ and  $\rho_{\za_{ij}}:\EZig\ze_i\to \EZig\ze_j,  \zv\mapsto (-1)^{|\zv|}\zv\za_{ij}$. 
Let $f:\Pi\Tiltz(0)=\Pi\zL(0)\,\into\, \Pi\Tiltz(1)=\EZig\ze_0$ be the embedding given by $\zv_0 \mapsto \zc_0$,
and let $g:\Pi\Tiltz(1)=\EZig\ze_0\,\onto\, \Pi\Tiltz(0)=\Pi\zL(0)$ be the surjection such that $\ze_0 \mapsto \zv_0$. Note that $f$ and $g$ are odd. 
Define the following elements of $\EZig'$:
\begin{enumerate}[\hspace{5mm}]
\item[$\bullet$] $\ze_i':=\pi_{\ell-i}$ for all $i\in I$;
\item[$\bullet$] $\zc_i':=\iota_{\ell-i}\circ \rho_{\zc_{\ell-i-1}}\circ \pi_{\ell-i}$ for all $i\in J$;
\item[$\bullet$] $\za_{i+1,i}':=\iota_{\ell-i}\circ \rho_{\za_{\ell - i -2, \ell - i - 1}}\circ\pi_{\ell-i-1}$ 
and $\za_{i,i+1}' :=\iota_{\ell-i-1}\circ \rho_{\za_{\ell - i - 1,\ell - i -2 }}\circ\pi_{\ell-i}$, for all $i=0,\dots,\ell-2$;

\item[$\bullet$] $\za_{\ell,\ell-1}':=\iota_1\circ f\circ \pi_0$ and $\za_{\ell-1,\ell}':=\iota_0\circ g\circ \pi_1$.
\end{enumerate}

\begin{Lemma}\label{LZ'} 
Mapping  $\ze_i\mapsto \ze_i',\ \za_{i,j}\mapsto \za_{i,j}',\ \zc_i\mapsto \zc_i'$ is an isomorphism of superalgebras $\EZig\iso \EZig'$. In particular, $\EZig$ is Ringel self-dual. 
\end{Lemma}
\begin{proof}
Using the fact that $\dim \Hom_{\EZig}(\De(i), \nabla(j)) = 0$ if $i > j$, and is equal to $1$ if $i = j$, it is easy to see that
$$ \{ \ze_i' \mid i \in I \} \sqcup \{ \za_{j, j+1}', \za_{j+1, j}', \zc_j' \mid j \in J \}$$
is a basis for $\EZig'$. Hence the given map is a linear isomorphism.

To check that the map is an algebra homomorphism, one proceeds by cases. For example, for $i \in J$, in $\EZig' = \End_{\EZig}(\Tiltz)^\sop$ we have
\begin{align*}
	\za_{i, i+1}' \cdot \za_{i+1, i}' &= - \za_{i+1, i}' \circ  \za_{i, i+1}'\\
	&= - \iota_{\ell-i}\circ \rho_{\za_{\ell - i -2, \ell - i - 1}}\circ\pi_{\ell-i-1} \circ \iota_{\ell-i-1}\circ \rho_{\za_{\ell - i - 1,\ell - i -2 }}\circ\pi_{\ell-i}\\
	&= - \iota_{\ell-i}\circ \rho_{\za_{\ell - i -2, \ell - i - 1}} \circ \rho_{\za_{\ell - i - 1,\ell - i -2 }}\circ\pi_{\ell-i}\\
	&=  \iota_{\ell-i} \circ \rho_{c_{\ell - i - 1}} \circ \pi_{\ell - i}\\
	&= \zc_i'.
\end{align*}
The other cases are checked similarly.
\end{proof}

We use the isomorphism $\EZig\iso \EZig'$ of Lemma~\ref{LZ'} to transport the heredity data $I,\zX,\zY$ from $\EZig$ onto a heredity data $I',\zX',\zY'$ for $\EZig'$ so that $I'=I$ {\em with the same order}, and 
$$\zX'(i):=
\left\{
\begin{array}{ll}
\{\ze_i',\za_{i-1,i}'\}   &\hbox{if $i>0$,}\\
\{\ze_0'\}  &\hbox{if $i=0$,}
\end{array}
\right.
\quad
\zY'(i):=
\left\{
\begin{array}{ll}
\{\ze_i',\za_{i,i-1}'\}   &\hbox{if $i>0$,}\\
\{\ze_0'\}  &\hbox{if $i=0$.}
\end{array}
\right.
$$
(Usually, one gets the opposite partial order on $I$ in this place but we have built in a relabeling $i\mapsto \ell-i$ into the construction). 
With this hereditary data, we have the right modules $\zL'(i),\De'(i),\nabla'(i)$ and $\Tiltz'(i)$. For example, $\Tiltz'(0)\simeq \zL'(0)$ (with $\ze_0'$ acting as identity and all the other standard generators acting as $0$) and $\Tiltz'(i)\simeq \Pi\ze'_{i-1}\EZig'$ for $i>0$. A routine check shows that, as a right $\EZig'$-module, $\Tiltz$ decomposes as follows:
$$
\Tiltz=\bigoplus_{i\in I}\Pi\Tiltz'(i),
$$
where the summands are defined explicitly as follows:
\begin{enumerate}[\hspace{5mm}]
\item[$\bullet$] $\Pi\Tiltz'(0)=\k\cdot \za_{\ell,\ell-1}\subseteq  \EZig\ze_{\ell-1}=\Pi\Tiltz(\ell)\subseteq \Tiltz$;
\item[$\bullet$] $\Pi\Tiltz'(\ell)=\spa_\k(\zv_0,\ze_0,\zc_0,\za_{0,1})\subseteq \Pi\Tiltz(0)\oplus \Pi\Tiltz(1)\oplus \Pi\Tiltz(2)\subseteq \Tiltz$ (dropping $\za_{0,1}$ if $\ell=1$);
\item[$\bullet$] for for $i\neq 0,\ell$, we set 
\begin{align*}
\Pi\Tiltz'(i)
&=\spa_\k(\ze_{\ell-i},\za_{\ell-i,\ell-i-1},\za_{\ell-i,\ell-i+1},\zc_{\ell-i})
\\
&\subseteq \Pi\Tiltz(\ell-i+1)\oplus \Pi\Tiltz(\ell-i)\oplus \Pi\Tiltz(\ell-i+2)\subseteq \Tiltz 
\end{align*}
(dropping $\za_{\ell-i,\ell-i+1}$ if $i=1$).
\end{enumerate}

In this last paragraph we suppose that $\k=R$ and recall the theory of \S\S\ref{SSGaA},\ref{SSGammaTildeMod}. 
We have the subalgebra $\a_R:=\spa(B_\a)=\spa_R(\ze_i\mid i\in I)\subseteq \EZig_{R,\0}$ and the analogous subalgebra $\a'_R=\spa_R(\ze_i'\mid i\in I)\subseteq \EZig'_{R,\0}$. We have $R$-module decomposition $\Tiltz_{R,\0}=\Tiltz_{R,\a}\oplus \Tiltz_{R,\c}$ where 
$\Tiltz_{R,\a}=\spa_R(\ze_j\mid j\in J)$ and $\Tiltz_{R,\c}=\spa_R(\zc_j\mid j\in J)$. Then it is clear from the explicit construction above that 
\begin{equation}\label{E170921}
\a_R\cdot(\Tiltz_{R,\a})\cdot\a'_R\subseteq \Tiltz_{R,\a}.
\end{equation} 
For any $i\in I$, $\Pi\Tiltz(i)_{R,\a}:=\Pi\Tiltz(i)_R\cap \Tiltz_{R,\a}$, $\Pi\Tiltz(i)_{R,\c}:=\Pi\Tiltz(i)_R\cap \Tiltz_{R,\c}$, $\Pi\Tiltz'(i)_{R,\a}:=\Pi\Tiltz'(i)_R\cap \Tiltz_{R,\a}$, $\Pi\Tiltz'(i)_{R,\c}:=\Pi\Tiltz'(i)_R\cap \Tiltz_{R,\c}$. These define calibrations on each $\Pi \Tiltz(i)_R$ and $\Tiltz_R$ as left $\EZig_R$-supermodules, as well as on each $\Pi \Tiltz'(i)_R$ and $\Tiltz_R$ as right $\EZig'_R$-supermodules.



\subsection{\boldmath A full tilting module for $T^Z(n,d)$}
\label{TZFull}
From now on until the end of the paper we fix $n\in\Z_{>0}$ and $d\in\Z_{\geq 0}$ such that $d\leq n$. 

To construct a full tilting module for based quasihereditary superalgebra $T^\EZig(n,d)$, we need to first work integrally. Recall from the previous subsection that $\Tiltz_R$ is 
a $(\EZig_R,\EZig'_R)$-bisupermodule.

Recall the constructions of \S\ref{SSMNA}. In particular, we have the left $M_n(\EZig_R)$-module structure on $\Co_n(\Tiltz_R)$ 
and a right $M_n(\EZig'_R)$-module structure on $\Ro_n(\Tiltz_R)$. 
Setting as in \S\ref{SSMNA}, $\Co_n(\Tiltz_R)_\a:=\Co_n(\Tiltz_{R,\a})$ and $\Ro_n(\Tiltz_R)_\a:=\Ro_n(\Tiltz_{R,\a})$, we have by 
(\ref{E170921}) that $M_n(\a)\Co_n(\Tiltz_R)_\a\subseteq \Co_n(\Tiltz_R)_\a$ and $\Ro_n(\Tiltz_R)_\a M_n(\a)\subseteq \Ro_n(\Tiltz_R)_\a$. So by Lemma~\ref{GaTiMod} (and its right module analogue), $\tilde\Ga^d \Co_n(\Tiltz_R)$ is a left module over $T^\EZig(n,d)_R$ and $\tilde\Ga^d \Ro_n(\Tiltz_R)$ is a right  module over $T^{\EZig'}(n,d)_R$. Similarly, for every $i\in I$, we have left $T^\EZig(n,d)_R$-modules $\tilde\Ga^d \Co_n(\Pi\Tiltz(i)_R)$ and right $T^{\EZig'}(n,d)_R$-modules $\Ro_n(\Pi\Tiltz'(i)_R)$. 
Extending scalars, we have a left supermodule 
$$\scrT^d_i:=\F\otimes_R \tilde\Ga^d \Co_n(\Pi\Tiltz(i)_R)$$ 
over $T^\EZig(n,d)=\F\otimes_RT^\EZig(n,d)_R$. 

Recall from \S\ref{SSQHT} that for $d\leq n$, the algebra $T^\EZig(n,d)$ is quasi-hereditary with respect to the poset $\La^I_+(n,d)$ with partial order $\leq_I$, so it has its own standard modules $\{\De(\bla)\mid \bla\in\La^I_+(n,d)\}$, costandard modules $\{\nabla(\bla)\mid \bla\in\La^I_+(n,d)\}$ and 
indecomposable tilting modules $\{\Tilt(\bla)\mid \bla\in\La^I_+(n,d)\}$. Moreover, by \cite[Proposition 6.20]{KMgreen3}, the anti-involution $\tau$ on $\EZig$ extends to the anti-involution 
$$\tau_{n,d}:T^\EZig(n,d)\to T^\EZig(n,d),\ \eta_{\br,\bs}^\bb\mapsto \eta_{\bs,\br}^{\tau(\bb)}$$ 
where for a $d$-tuple $\bb=\zb_1\cdots\zb_d\in\zB^d$ we denote $\tau(\bb):=\tau(\zb_1)\cdots\tau(\zb_d)$. As in \cite[(2.14)]{KW1}, we then have for all $\bla\in \La^I_+(n,d)$:
\begin{equation}\label{ETauDual}
\De(\bla)^{\tau_{n,d}}\simeq \nabla(\bla).
\end{equation}

Since $\tau_{n,d}(\eta_\bmu)=\eta_\bmu$ for all $\bmu\in\La^I(n,d)$, we deduce:

\begin{Lemma} \label{LDeltaNablaChar}
For all $\bla\in\La^I_+(n,d)$, we have $\ch \De(\bla)=\ch \nabla(\bla)$. 
\end{Lemma}

\begin{Proposition}\label{TZTilt}
For all $i\in I$, the left $T^\EZig(n,d)$-module $\scrT^d_i$ is tilting and has a unique maximal weight $\biota_i(1^d)$ (with respect to $\leq_I$). 
\end{Proposition}
\begin{proof}
Suppose first that $i=0$. Since $\biota_0 (1^d)$ is minimal in $\La^+_I(n, d)$ it follows that $\Tilt (\biota_0(1^d)) \simeq \De (\biota_0(1^d))\simeq L (\biota_0(1^d))$.
Using the assumption $d\leq n$, we note that $\biota_0(1^d)$ is the only weight of $\scrT^d_0$ that lies in $\La^+_I(n, d)$, and the corresponding weight multiplicity is $1$, so 
\begin{equation}\label{ETilt0}
	\scrT^d_0 \cong L (\biota_0(1^d))\simeq \Tilt(\biota_0((1^d)))\simeq \De (\biota_0((1^d))). 
\end{equation}

Let now $i > 0$. Since $\Pi\Tiltz(i)_R=\EZig_R\ze_{i-1}$, we have $\Co_n(\Pi\Tiltz(i)_R)\simeq M_n(\EZig_R)\xi_{1,1}^{\ze_{i-1}}$, and so 
by Lemma~\ref{GaCnIdem} and extension of scalars, we have 
\begin{equation}\label{ETiltNonZero}
	\scrT^d_i \simeq T^\EZig(n,d) \eta^{\ze_{i-1}^d}_{1^d ,1^d}=
	T^\EZig(n,d) \eta_{\biota_{i-1}(d)}.
\end{equation}
In particular, $\scrT^d_i$ is projective, and thus has a standard filtration. To prove that $\scrT^d_i$ also has a costandard filtration, it suffices to show that it is $\tau_{n,d}$-self-dual, or equivalently possesses a non-degenerate $\tau_{n,d}$-contravariant bilinear form. 

To construct this form we work over $R$. 
Recall that for $i>1$, we have $\Pi\Tiltz(i)_R = \EZig_R \ze_{i-1}$ has basis $\zB^{\Pi\Tiltz(i)} := \{ \ze_{i-1}, \zc_{i-1}, \za_{i-2, i-1}, \za_{i, i-1} \}$.
Consider the bilinear form $( \cdot, \cdot)$ on $\Pi\Tiltz(i)_R$ such that 
$$(\ze_{i-1}, \zc_{i-1}) = - (\zc_{i-1},\ze_{i-1})= (\za_{i-2, i-1}, \za_{i-2, i-1}) = (\za_{i, i-1}, \za_{i, i-1}) = 1,$$
and all the other pairings of basis elements are $0$. Note that this form is even, non-degenerate, $\tau$-contravariant, and superantisymmetric.
Extending this form in the obvious way to $\Co_n (\Pi\Tiltz(i)_R)$ results in a non-degenerate, superantisymmetric, $\tau_{n,1}$-contravariant form again denoted $(\cdot, \cdot)$. Note that this form satisfies the assumptions of 
Lemma ~\ref{GaTiForm}. Applying Lemma ~\ref{GaTiForm}, we obtain an even, non-degenerate, $\tau_{n,d}$-contravariant form on $\tilde \Ga^d \Co_n ( \Pi\Tiltz(i)_R)$. 
Extending scalars, we deduce that, $\scrT^d_i$ is $\tau_{n,d}$-self-dual.

Denote by $v_r^\zb\in \Co_n (\Pi\Tiltz(i)_R)$ the column vector with $\zb\in \Pi\Tiltz(i)_R$ in the $r$th position and $0$s elsewhere. 
To see that the unique maximal weight of $\scrT^d_i$ is $\biota_i(1^d)$, 
observe that the unique maximal weight vector of $\Pi\Tiltz(i)_R$ is $\za_{i, i-1}$, which is odd, and so the vector $v_1^{\za_{i, i-1}}*\cdots*v_d^{\za_{i, i-1}}\in \tilde \Ga^d \Co_n(\Pi\Tiltz(i)_R)$ 
has weight $\biota_i(1^d)$ and all other weight vectors appearing in $\tilde \Ga^d \Co_n(\Pi\Tiltz(i)_R)$ have smaller weight.

The case $i=1$ is similar to the case $i>1$ but $\Pi\Tiltz(1)_R = \EZig_R \ze_{0}$ has basis $\zB^{\Pi\Tiltz(1)} := \{ \ze_0, \zc_0, \za_{1,0} \}$, and we use the form such that 
$$(\ze_0, \zc_0) = - (\zc_0,\ze_0) = (\za_{1, 0}, \za_{1, 0}) = 1$$
are the only non-trivial pairings of basis elements. 
\end{proof}

Recalling that $\Tiltz_R$ is a $(\EZig_R,\EZig'_R)$-bisupermodule, we can now consider $M_n(\Tiltz_R)$ as an $(M_n(\EZig_R),M_n(\EZig'_R))$-bisupermodule in the obvious way. Take 
$M_n(\Tiltz_R)_\a:=M_n(\Tiltz_{R,\a})$ and $M_n(\Tiltz_R)_\c:=M_n(\Tiltz_{R,\c})$.  
In view of Lemma~\ref{GaTiMod} (and its right module analogue), $\tilde\Ga^dM_n (\Tiltz_R)$ is a $(T^\EZig(n,d)_R,T^{\EZig'}(n,d)_R)$-bisupermodule. 
We now extend the scalars from $R$ to $\F$ to get the $(T^\EZig(n,d),T^{\EZig'}(n,d))$-bisupermodule 
$${\mathscr T}:=\F\otimes_R \tilde\Ga^dM_n (\Tiltz_R).$$

For each composition $\mu \in \La(n,d)$ define $\scrT_i^{\mu} := \scrT_i^{\mu_1} \otimes \cdots \otimes \scrT_i^{\mu_n}$. 
Furthermore, for each multicomposition $\bmu=(\mu^{(0)},\dots,\mu^{(\ell)}) \in \La^I(n,d)$ define $\scrT^{\bmu} := \scrT_0^{\mu^{(0)}} \otimes \cdots \otimes \scrT_\ell^{\mu^{(\ell)}}$.

Since $\Tiltz_R = \bigoplus_{i\in I} \Pi \Tiltz(i)_R$, 
as left supermodules over $M_n(\EZig_R)$, we have 
$$
M_n(\Tiltz_R)\simeq\Co_n(\Tiltz_R)^{\oplus n}\simeq\bigoplus_{i\in I}\Co_n(\Pi\Tiltz(i)_R)^{\oplus n}.
$$
Now, using Lemmas~\ref{GaTiSum} and extending scalars,
we have as left $T^\EZig(n,d)$-supermodules:
\begin{equation}\label{ScrTDecomp}
\scrT \simeq \bigoplus_{\bmu \in \La^I(n,d)}\scrT^{\bmu}.
\end{equation}

For $\bla=(\la^{(0)},\dots,\la^{(\ell)})\in\La^I_+(n,d)$, we define the {\em conjugate multipartition} 
\begin{equation}\label{EBla'}
\bla' :=((\la^{(0)})',\dots,(\la^{(\ell)})')\in\La^I_+(n,d).
\end{equation}

\begin{Proposition} \label{PFullTilt}
As a left $T^\EZig(n,d)$-supermodule, ${\mathscr T}$ is a full tilting supermodule. 
\end{Proposition}
\begin{proof}
Note that each $\scrT^{\bmu}$
is tilting by Proposition~\ref{TZTilt} and Corollary~\ref{CTensTilt}. So $\scrT$  is tilting by (\ref{ScrTDecomp}). 
To show that $\scrT$ is full tilting, it suffices for each
$\bla \in \La^+_I(n, d)$ to find a summand 
$\scrT^{\bmu}$ in (\ref{ScrTDecomp}) which has maximal weight $\bla$. Fix $\bla=(\la^{(0)},\dots,\la^{(\ell)}) \in \La^+_I(n, d)$ and
take $\bmu = \bla'$. By Proposition~\ref{TZTilt} again, $\scrT_i^{s}$ has unique maximal weight $\biota_i(1^{s})$ for each $s \in \Z_{> 0}$. So 
$$\sum_{i=0}^\ell\sum_{r=1}^n\biota_i(1^{\mu^{(i)}_r}) = \sum_{i=0}^\ell\biota_i(\la^{(i)}) =\bla$$
is the unique maximal weight of $\scrT^{\bmu}$. 
\end{proof}

\begin{Corollary} \label{Faith}
As a left $T^\EZig(n,d)$-supermodule and as a right $T^{\EZig'}(n,d)$-supermodule, ${\mathscr T}$ is faithful. 
\end{Corollary}
\begin{proof}
As a left $T^\EZig(n,d)$-supermodule, ${\mathscr T}$ is faithful since it is full tilting by Proposition~\ref{PFullTilt} (a full tilting supermodule is faithful for example by \cite[Lemma 6]{Ringel}). The second statement follows similarly from the right module analogue of that proposition. 
\end{proof}

\subsection{\boldmath $T^Z(n,d)$ is Ringel self-dual}
\label{TZSelf}
In view of Corollary~\ref{Faith}, we have an embedding of $T^{\EZig'}(n,d)$ into $\End_{T^{\EZig}(n,d)} (\scrT)^\sop$. To prove that this embedding is an isomorphism, we now compute the dimension of $\End_{T^{\EZig}(n,d)} (\scrT)$.

Recalling (\ref{EAl}), for $\bla\in\La^I_+(n,d)$ and  $\bmu \in \La^I(n,d)$, let 
$$k_{\bla, \bmu}:=|\Std^\zX(\bla,\bmu)|.$$
By (\ref{EBasisDe}), (\ref{EWtvT}) and Lemma~\ref{LDeltaNablaChar}, we have 
\begin{equation}\label{E091221}
k_{\bla, \bmu}=\dim\eta_\bmu\De(\bla)=\dim\eta_\bmu\nabla(\bla). 
\end{equation}

Let $i \in I$. If $i\neq 0$, we define
\[ \bbeta_i(d,s) = \biota_{i-1}((s)) + \biota_i((1^{d-s})) \in \La^I_+(n,d)
\]
for all $0\leq s \leq d$
We also define 
$$
\bbeta_0(d,0):=\biota_0((1^{d})). 
$$
We define by $\Xi_{d,i}$ to be the set of all $\bbeta_i(d,s)$'s, i.e.
$$\Xi_{d,i}:=
\left\{
\begin{array}{ll}
\{ \bbeta_i(d,s)\mid 0\leq s\leq d\} &\hbox{if $i\neq 0$,}\\
\{\bbeta_0(d,0)\} &\hbox{if $i=0$.}
\end{array}
\right.
$$

\begin{Lemma} \label{1DKostka} 
Let $\bbeta \in \La^I_+(n,d)$ and $i\in I$. 
Then 
\[ (\scrT_i^d : \De (\bbeta) ) = \begin{cases}
	1 &\text{ if } \bbeta \in \Xi_{d,i},\\
	0 &\text{ otherwise}.\\
\end{cases}\]
\end{Lemma}
\begin{proof}
By (\ref{ETilt0}), we have $\scrT_0^d \cong \De(\biota_0((1^d)))$, so we may assume that $i \neq 0$. Then by (\ref{ETiltNonZero}), we have $\scrT_i^d \simeq T^\EZig(n,d) \eta_{\biota_{i-1}(d)}$. 
Now, using (\ref{EDeMult}) and (\ref{E091221}), we get 
\begin{align*}
	(\scrT_i^d : \De (\bbeta)) &= \dim \Hom_{T^\EZig(n,d)}(\scrT_i^d, \nabla (\bbeta)) 
	\\
	&= \dim \Hom_{T^\EZig(n,d)}( T^\EZig(n,d)\eta_{\biota_{i-1}(d)}, \nabla (\bbeta))
	\\
	&= \dim \eta_{\biota_{i-1}(d)}  \nabla (\bbeta) 
	\\
	&= k_{\bbeta, \biota_{i-1}(d)}. 
\end{align*}
It remains to observe that $k_{\bbeta, \biota_{i-1}(d)}=1$ if $\bbeta = \bbeta_i(d,s)$ for some $0 \leq s \leq d$ and $k_{\bbeta, \biota_{i-1}(d)}=0$ otherwise.
\end{proof}

Let $0\leq r\leq d$. Recalling (\ref{EBoldLR}), our next goal is to compute the Littlewood-Richardson coefficient $c^\bla_{\balpha,\bbeta}$ for all $\bla\in\La^I_+(n,d)$, $\balpha\in \La^I_+(n,d-r)$ and $\bbeta\in \Xi_{r,i}$. Let $i\in I$ and $\bbeta=\bbeta_i(r,s)\in \Xi_{r,i}$, in particular, $0\leq s\leq r$, and $s=0$ if $i=0$. We define 
$\Om^{\bla}_\bbeta$ to be the set of all $\balpha = (\al^{(0)}, \ldots, \al^{(\ell)}) \in \La^I_+(n, d-r)$ such that 
$\al^{(j)} = \la^{(j)}$ for all $j \notin \{i-1, i\}$,
$[\al^{(i-1)}]$ is obtained from $[\la^{(i-1)}]$ by removing $s$ boxes from distinct columns, and
$[\al^{(i)}]$ is obtained from $[\la^{(i)}]$ by removing $r-s$ boxes from distinct rows (if $i=0$, then the condition on $[\al^{(i-1)}]$ should be dropped).

\begin{Lemma}\label{LRBeta}
Let $0\leq r\leq d$, $i\in I$, $\bla\in\La^I_+(n,d)$, $\balpha\in \La^I_+(n,d-r)$ and $\bbeta= \Xi_{r,i}$. 
Then
\[ c_{\balpha, \bbeta}^{\bla} = \begin{cases}
	1 &\text{ if } \balpha \in \Om^{\bla}_\bbeta\\
	0 &\text{ otherwise.}\\
\end{cases}\]
\end{Lemma}
\begin{proof}
The result is an immediate consequence of the Littlewood-Richardson rule.
\end{proof}

For each $\bmu \in \La^I(n,d)$, define $\rev{\bmu}=(\rev{\mu}^{(0)},\dots,\rev{\mu}^{(\ell)}) \in \La^I(n,d)$ by setting $\rev{\mu}^{(i)}:=\mu^{(\ell - i)}$ for all $i\in I$. Recall (\ref{EBla'}).

\begin{Proposition} \label{GenKostka}
Let $\bla \in \La^I_+(n,d)$ and $\bmu \in \La^I(n,d)$. Then 
$
(\scrT^\bmu : \De (\bla)) = k_{\rev{\bla}', \rev{\bmu}}\,.
$
\end{Proposition}
\begin{proof}
We proceed by induction on the number of non-zero parts of $\bmu$. To start the induction, we suppose that $\bmu$ has only one row, in which case $\scrT^\bmu\cong \scrT_i^d$ for some $i$, and the result follows from Lemma~\ref{1DKostka}. So we may assume that $\bmu$ has at least two rows.

Let now $i$ be maximal such that $\mu^{(i)} \neq \varnothing$ and pick the largest $t$ such that $\mu^{(i)}_t \neq 0$. Denote  $r := \mu^{(i)}_t$. Let $\nu^{(i)}$ be $\mu^{(i)}$ with last non-zero row removed: 
$$\nu^{(i)}:=(\mu^{(i)}_1,\dots,\mu^{(i)}_{t-1},0,\dots,0)\in\La(n,|\mu^{(i)}|-r),$$ 
and $\nu^{(j)}:=\mu^{(j)}$ for all $j\neq i$.
Set 
$$\bnu:=(\nu^{(0)},\dots,\nu^{(\ell)})\in\La^I(n,d-r).$$ 
Then $\scrT^{\bmu} \cong \scrT^{\bnu} \otimes \scrT_i^r$.

By (\ref{EXY}), we have $\zX(0) = \{ \ze_0\}$ and $\zX(i) = \{ \ze_i, \za_{i-1,i}\}$ for  $i \neq 0$. Recalling \S\ref{SSComb}, 
for $i \neq 0$, we put the total order on $\Alph_{\zX(i)}$ given by: $$1^{\ze_i} < \cdots < n^{\ze_i}< 1^{\za_{i-1,i}} < \cdots < n^{\za_{i-1, i}}.$$
And we endow $\Alph_{\zX(0)}$ with the order $1^{\ze_0} < \cdots < n^{\ze_0}$.

By \cite[Corollary (4.13)]{KW1}, inductive hypothesis, Lemma~\ref{1DKostka}, and Lemma~\ref{LRBeta} we have 
\begin{align*}
	(\scrT^\bmu : \De (\bla)) 
	&=
	(\scrT^{\bnu} \otimes \scrT_i^r: \De (\bla))
	\\
	&= \sum_{\balpha\in \La^I_+(n,d-r)}\sum_{\bbeta\in\La^I_+(n,r)} c_{\balpha, \bbeta}^\bla  (\scrT^{\bnu} : \De(\balpha)) (\scrT_i^r: \De(\bbeta)) 
	\\
	&= \sum_{\balpha\in \La^I_+(n,d-r)}\sum_{\bbeta\in\Xi_{r,i}}c_{\balpha, \bbeta}^\bla k_{\rev\balpha', \rev\bnu}
	\\
	&= \sum_{\bbeta\in\Xi_{r,i}}\sum_{\balpha\in \Om^{\bla}_\bbeta} k_{\rev\balpha', \rev\bnu}
	\\
	&= \sum_{\bbeta\in\Xi_{r,i}}\sum_{\balpha\in \Om^{\bla}_\bbeta} |\Std^\zX(\rev\balpha',\rev\bnu)|.
\end{align*}
Since $k_{\rev{\bla}', \rev{\bmu}}=|\Std^\zX(\rev{\bla}', \rev{\bmu})|$ it remains to prove that there is a bijection 
$$
\bigsqcup_{\bbeta\in\Xi_{r,i}}\bigsqcup_{\balpha\in \Om^{\bla}_\bbeta} \Std^\zX(\rev\balpha',\rev\bnu)
\iso 
\Std^\zX(\rev{\bla}', \rev{\bmu}). 
$$

Let $\bbeta\in\Xi_{r,i}$, $\balpha\in \Om^{\bla}_\bbeta$ and $\T\in \Std^\zX(\rev\balpha',\rev\bnu)$. By definition, $\bbeta$ is of the form $\bbeta_i(r,s)$. Moreover, the Young diagram $[\balpha]$ is obtained by removing $s$ nodes from distinct columns of the 
$(i-1)$st component 
$[\la^{(i-1)}]$ and $r-s$ nodes from distinct rows of the $i$th component $[\la^{(i)}]$ of $[\bla]$. Therefore $[\rev{\balpha}']$ is obtained by removing $s$ nodes $N_1,\dots,N_s$ from distinct rows of the $(\ell-i+1)$st component of $[\rev{\bla}']$ and $r-s$ nodes $M_1,\dots,M_{r-s}$ from distinct columns of the $(\ell-i)$th component of $[\rev{\bla}']$. Now extend $\T$ to the tableau $\hat\T\in \Std^\zX(\rev{\bla}', \rev{\bmu})$ by setting 
$$
\hat \T(N_1)=\dots=\hat \T(N_s)=t^{\za_{\ell - i, \ell - i + 1}}\qquad\text{and}\qquad 
\hat \T(M_1)=\dots=\hat \T(M_{r-s})=t^{\ze_{\ell -i}}.
$$
The tableaux $\hat\T$ is indeed standard since, by maximality of $i$ and $t$, we have $\T(N)<t^{\za_{\ell - i, \ell - i + 1}}$ for all $N$ in the $(\ell-i+1)$st component of $[\rev{\balpha}']$  
and $\T(N)<t^{\ze_{\ell -i}}$ for all $N$ in the $(\ell-i)$th component of $[\rev{\balpha}']$. The map $\T\mapsto \hat\T$ is clearly injective. To see that it is surjective, it suffices to show that for any $\Stab\in \Std^\zX(\rev{\bla}', \rev{\bmu})$ there exists $\bbeta\in\Xi_{r,i}$ and $\balpha\in\Om^\bla_\bbeta$ with  
$$
[\rev{\bla}']\setminus\{N\mid \Stab(N)\in\{t^{\ze_{\ell -i}},t^{\za_{\ell - i, \ell - i + 1}}\}\}=[\rev{\balpha}'].
$$

Indeed, there are exactly $\rev{\mu}^{(\ell-i)}_t=\mu^{(i)}_t=r$ nodes $N$ in the Young diagram $[\rev{\bla}']$ such that $\Stab(N)\in \{t^{\ze_{\ell -i}},t^{\za_{\ell - i, \ell - i + 1}}\}$. 
So for some $0\leq s\leq r$, we can write  
\begin{align*}
	\{N\in[\rev{\bla}']\mid  \Stab(N)=t^{\za_{\ell - i, \ell - i + 1}}\}&=\{N_1,\dots,N_s\},\\ 
	\{N\in[\rev{\bla}']\mid  \Stab(N)=t^{\ze_{\ell - i}}\}&=\{M_1,\dots,M_{r-s}\}. 
\end{align*}
By maximality of $i$ and $t$, we have that the 
nodes 
$N_1,\dots,N_s$ are in the ends of distinct rows of the $(\ell-i+1)$st component of $[\rev{\bla}']$ and the nodes $M_1,\dots,M_{r-s}$ are in the ends of distinct columns of the $(\ell-i)$th component of $[\rev{\bla}']$. It remains to note that removing these nodes produces a shape $[\rev{\balpha}']$ for $\balpha\in\Om^\bla_{\bbeta_i(r,s)}$. 
\end{proof}

\begin{Theorem}\label{TZigRingel}
Let $d\leq n$. We have $\End_{T^{\EZig}(n,d)} (\scrT)^\sop \cong T^{\EZig'}(n,d)$. In particular, $T^{\EZig}(n,d)$ is Ringel self-dual. 
\end{Theorem}
\begin{proof}
In view of Lemma~\ref{LZ'}, we have that $T^\EZig(n,d)\cong T^{\EZig'}(n,d)$. In particular the second statement of the theorem follows from the first one. 

By Corollary~\ref{Faith}, $T^{\EZig'}(n,d)$ embeds into $\End_{T^{\EZig}(n,d)}(\scrT)^\sop$. So it suffices to show  that 
$\dim\End_{T^{\EZig}(n,d)}(\scrT)=\dim T^{\EZig}(n,d)$. 

In view of (\ref{ETauDual}), we have that each $\scrT^\bmu$ is $\tau_{n,d}$-self-dual and $(\scrT^\bmu : \De (\bla)) = (\scrT^\bmu : \nabla(\bla))$ for all $\bla \in \La^I_+(n,d)$. We now have: 
\begin{align*}
	\dim \End_{T^{\EZig}(n,d)}(\scrT) &= \sum_{\bmu, \bnu \in \La^I(n,d)} \dim \Hom_{T^\EZig(n,d)}(\scrT^\bmu, \scrT^\bnu)\\
	&= \sum_{\bla \in \La^I_+(n,d)}\sum_{\bmu, \bnu \in \La^I(n,d)} (\scrT^\bmu : \De(\bla)) (\scrT^\bnu: \nabla(\bla))\\
	&= \sum_{\bla \in \La^I_+(n,d)}\sum_{\bmu, \bnu \in \La^I(n,d)} (\scrT^\bmu : \De(\bla)) (\scrT^\bnu: \De(\bla))\\
	&= \sum_{\bla \in \La^I_+(n,d)}\sum_{\bmu, \bnu \in \La^I(n,d)} k_{\rev\bla', \rev\bmu} k_{\rev\bla', \rev\bnu}\\
	&= \sum_{\bla \in \La^I_+(n,d)}\sum_{\bmu, \bnu \in \La^I(n,d)} k_{\bla, \bmu} k_{\bla, \bnu}
	\\
	&=\dim T^{\EZig}(n,d),
\end{align*}
where we have used (\ref{ScrTDecomp}) for the first equality, \cite[Proposition A2.2]{DonkinQS} for the second equality, Proposition~\ref{GenKostka} for the fourth equality and 
\cite[Theorem 5.17]{KMgreen3} for the last equality. 
\end{proof}


\end{document}